%% file: main.tex
\documentclass[UTF8,oneside,10pt]{article}
\input{preamble}

\begin{document}
% \maketitle

\begin{center}
{\Large \textbf{Uniqueness of Maximal Curve Systems on Punctured Projective Planes}}

{\Large \textbf{ }} 

{\large \textsc{Xiao} CHEN$^{*}$ and \textsc{Wujie} SHEN$^{**}$} 

{\small Tsinghua University}

{\footnotesize \emph{e-mail:} x-chen20@mails.tsinghua.edu.cn$^{*}$}

{\footnotesize \emph{e-mail:} shenwj22@mails.tsinghua.edu.cn$^{**}$}

\end{center}

\begin{abstract}

    We show that the maximal $1$-system of loops in a punctured projective plane is unique up to the mapping class group action if and only if the number of punctures is at most five.

\end{abstract}

% \keywords{Systems of curves; arc complexes; hyperbolic geometry; mapping class groups; Teichmüller spaces.}

% \newpage
%\tableofcontents
% \newpage

%==============================%
%           正文内容           %
%------------------------------%

% \label{}
% \cref{}
%\begin{multicols}{2}

\section{Introduction}\label{section1}

The first listed author \cite{chen2024systems} generalized the curve complex \cite{harvey1981boundary} by introducing the \emph{(complete) $k$-curve complex}, constructed from \emph{(complete) $k$-systems of loops} \cite{juvan1996systems}. This paper focuses on complete $1$-curve complexes. This simplicial complex is defined so that each cell corresponds to the collection of isotopy classes of simple loops forming a complete $1$-system of loops, that is: these loops lie in distinct free isotopy classes and pairwise intersect exactly once. A system of loops is called \emph{maximal} if it is maximal by inclusion (as a set), or equivalently if it corresponds to a cell of maximal dimension in the complex.

The primary aim of this paper is to characterise when maximal complete $1$-systems of loops on the $n$-punctured projective plane $N_{1,n}$ are unique up to the action of the mapping class group.

\begin{thm}
\label{main}
Maximal complete $1$-systems of loops on $N_{1,n}$ are unique up to the action of the mapping class group if and only if $n \leq 5$.
\end{thm}

One motivation for studying the action of the mapping class group on maximal complete $1$-curve complexes comes from its conceptual proximity to the corresponding action on pants complexes \cite{MR579573}, which serve as a powerful tool in the study of the Weil--Petersson geometry of moduli spaces \cite{MR1940162,MR2262875,MR2399656,MR2872556,lackenby2024bounds}. Another motivation comes from number theory. Maximal complete $1$-systems of loops on the once-punctured torus are in natural bijection with the positive integer solutions to the (ternary) Markoff equation $x_1^2 + x_2^2 + x_3^2 = x_1 x_2 x_3$ \cite{zbMATH03108522}. This correspondence extends to a relation between the positive integer solutions of the (quaternary) Markoff--Hurwitz equation $x_1^2 + x_2^2 + x_3^2 + x_4^2 = x_1 x_2 x_3 x_4$ and maximal complete $1$-systems of loops on the thrice-punctured projective plane \cite{huang2017simple}. The combinatorics of such systems exhibit a remarkable symmetry: all maximal complete $1$-systems are equivalent up to the action of the mapping class group. This property holds for both the once-punctured torus and the thrice-punctured projective plane. This naturally leads us to ask whether such symmetry holds more generally, e.g.: on $n$-punctured projective planes.

A key prerequisite for understanding mapping class group orbits of maximal complete $1$-systems is the determination of their maximal cardinality. Malestein--Rivin--Theran \cite{malestein2014topological} showed that any maximal complete $1$-system on a closed orientable surface consists of exactly $2g+1$ loops\footnote{This cardinality also holds for orientable surfaces with boundary, as shown in \cite[Proposition 3.2]{chen2024systems}.}. With the maximal cardinality determined, Aougab \cite{aougab2014large} then proved that on orientable closed surfaces, such a system is unique up to the action of the mapping class group if and only if the genus is less than three. More generally, Aougab and Gaster \cite{aougab2017curves} established quantitative bounds for the number of orbits of maximal complete $1$-systems under the action of the mapping class group on orientable closed surfaces.

For non-orientable surfaces, the first listed author \cite{chen2024systems} proved that the cardinality of maximal complete $1$-systems of loops grows on the order of $|\chi|^{2}$, where $\chi$ is the Euler characteristic of the underlying surface. In particular, he derived the cardinality for $n$-punctured projective planes as $\frac{1}{2}n(n-1)+1$. This result paves the way for our present investigation.

\section{Background}\label{section2}

\subsection{Curves on surfaces}

\begin{defi}[Simple curves]

We refer to both loops and arcs as \emph{curves} and will often conflate a curve with its image. A \emph{simple loop} on a surface $F$ is defined by an embedding map $\gamma : \mathbb{S}^1 \hookrightarrow F$. If $F$ is a surface with non-empty boundary $\partial F$, then we define a \emph{simple arc} on $F$ as an embedding $\alpha : [0,1] \hookrightarrow F$ such that $\alpha (\partial[0,1]) \subset \partial F$.
\end{defi}

\begin{nota}
\label{nata: B is what}
We let $\gamma$ represent a loop, let $\alpha$ represent an arc and let the symbol $\beta$ represent a curve. Let $B = \mathbb{S}^1$ if $\beta$ is a simple loop and $B = [0,1]$ if $\beta$ is a simple arc.
\end{nota}

\begin{defi}[Regular neighbourhoods of curves]
\label{defi RN}

Let $\Omega$ be a finite collection of curves on $F$. A regular neighbourhood $W(\Omega)$ of $\Omega$ is a locally flat, closed subsurface of $F$ containing $\cup_{\beta \in \Omega} \beta$ such that there is a strong deformation retraction $H: W(\Omega) \times I \rightarrow W(\Omega)$ onto $\cup_{\beta \in \Omega} \beta$ where $H\vert_{(W(\Omega) \cap {\partial F}) \times I}$ is a strong deformation retraction onto $\cup_{\beta \in \Omega}   \beta \cap {\partial F}$.

\end{defi}

\begin{defi}[Free isotopies]

Let $\beta_0$ and $\beta_1$ be two curves, either both arcs or both loops. We say $\beta_0$ is \emph{freely isotopic} to $\beta_1$, if there is a collection of curves $\left\{ \beta(t,\cdot) \mid t \in [0,1] \right\}$, all of which are arcs or all loops, such that $\beta(0,\cdot) = \beta_0 , \beta(1,\cdot) = \beta_1$ and $\beta(\cdot,\cdot): [0,1] \times B \rightarrow F$ is a continuous map. We denote the \emph{free isotopy class} by $H := [\beta_0] = [\beta_1]$.

In the case where $\beta_0$ and $\beta_1$ are loops, $H$ is the usual free isotopy class of loops. If $\beta_0$ and $\beta_1$ are arcs, then $H$ refers to the isotopy class of arcs whose endpoints remain on their respective boundary components throughout the isotopy.

\end{defi}

\begin{defi}[Essential curves]

We say that a simple loop is \emph{essential} if it is not homotopically trivial, cannot be isotoped into any boundary component, and is primitive. We say that a simple arc is essential if it is not isotopic to any arc in any boundary component.

\end{defi}

\begin{defi}[$1$-sided vs. $2$-sided loops]

We say that a simple loop $\gamma$ is \emph{$1$-sided} if the regular neighbourhood of $\gamma$ is homeomorphic to a Möbius strip. We say that $\gamma$ is \emph{$2$-sided} if the regular neighbourhood of $\gamma$ is homeomorphic to an annulus.

Since the number of boundary components of the regular neighbourhood of a loop does not change under isotopy, $1$-sidedness and $2$-sidedness are well-defined for free isotopy classes.

\begin{figure}[H]
\centering
\includegraphics[width=\textwidth]{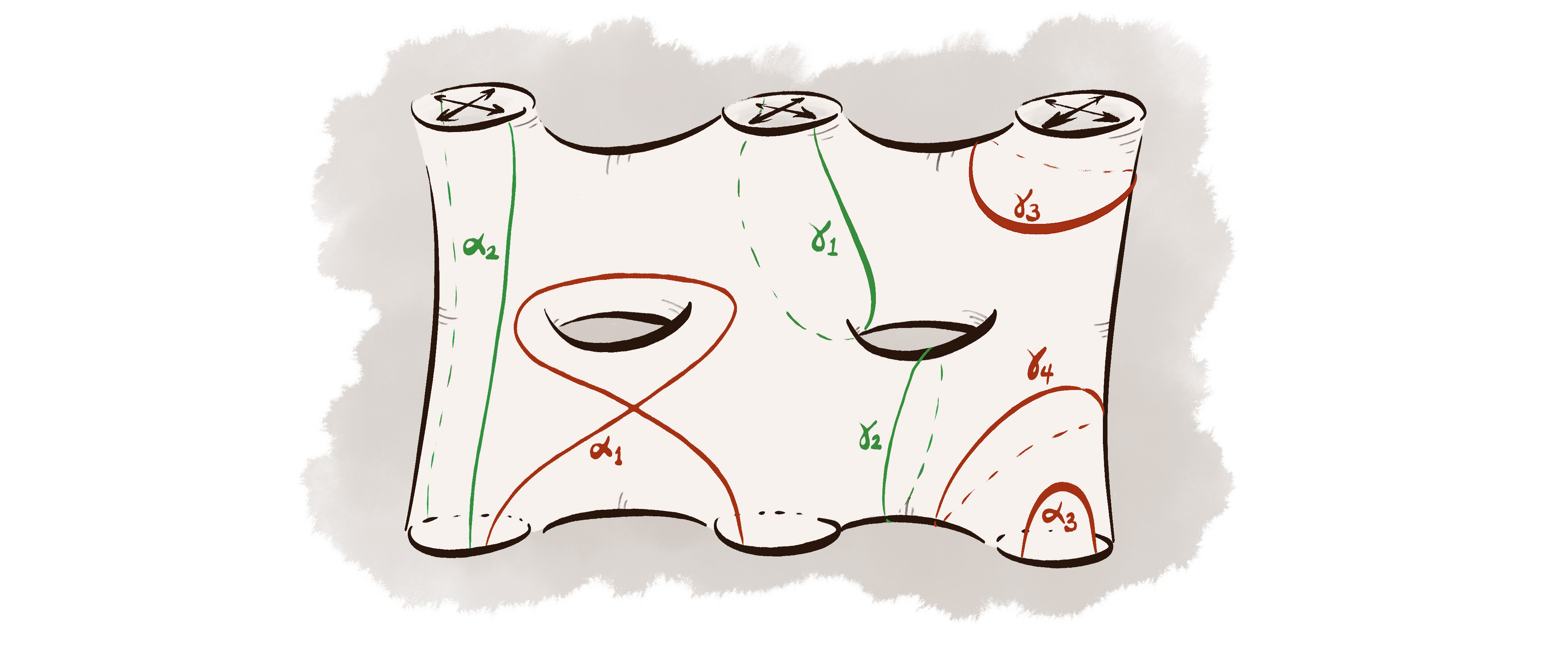}
\caption{Here are some examples of simple essential curves (in green) and non-simple or non-essential curves (in red) on $N_{7,3}$. The curve $\alpha_1$ is a non-simple arc, $\alpha_2$ is an essential simple arc, $\gamma_1$ is an essential $1$-sided simple loop, $\gamma_2$ is a essential $2$-sided simple loop, $\alpha_3$ is a non-essential simple arc and $\gamma_3, \gamma_4$ are two non-essential simple loops, where $\gamma_3$ is non-primitive.}
\label{CurvesExamples}
\end{figure}

\end{defi}

\subsection{Systems of curves}

\begin{defi}[Geometric intersection numbers]

Given two curves $\beta_1: B_1 \rightarrow F$, $\beta_2: B_2 \rightarrow F$ (see \cref{nata: B is what}), we define the \emph{geometric intersection number} $i (\beta_1 , \beta_2)$ as
\begin{align}
i (\beta_1 , \beta_2) := \# \left\{ (b_1, b_2) \in B_1 \times B_2 \middle| \beta_1(b_1) = \beta_2(b_2) \right\} . \notag
\end{align}

Consider two free isotopy classes (not necessarily distinct) $[\beta_1] , [\beta_2]$ of two simple curves $\beta_1, \beta_2$ on $F$. We define the \emph{geometric intersection number} of $[\beta_1]$ and $[\beta_2]$ as follows:
\begin{align}
i ([\beta_1] , [\beta_2]) := \min \left\{ i (\bar{\beta}_1 , \bar{\beta}_2) \middle|  \bar{\beta}_1 \in [\beta_1] , \bar{\beta}_2 \in [\beta_2] \right\}. \notag
\end{align}

\end{defi}

\begin{rmk}
We say two curves are \emph{transverse} if they intersect transversely at all intersection points. 
\end{rmk}

\begin{defi}[Systems of curves]
\label{defi: Systems of curves}

A \emph{system} $\Omega$ of curves on $F$ is either a collection $\Omega = L$ of simple loops or a collection $\Omega = A$ of simple arcs, such that the curves in $\Omega$ are essential, any two distinct curves $\beta_1, \beta_2$ are \emph{transverse, non-isotopic, and are in minimal position} (that is $i(\beta_1, \beta_2) = i([\beta_1], [\beta_2])$).

\end{defi}

\begin{defi}[$k$-systems of curves]
\label{defi: $k$-systems of curves}

We call a system of curves a \emph{$k$-system of curves} on $F$ if the geometric intersection number of any pair of elements in the system is at most $k$. Let $\mathscr{L}(F,k)$ be the collection of all $k$-systems of loops on $F$, and let $\mathscr{A}(F,k)$ be the collection of all $k$-systems of arcs on $F$.

\end{defi}

\begin{defi}[Complete $k$-systems of curves]

\label{defi: Complete $k$-systems of curves}

We call a system of curves a \emph{complete $k$-system of curves} on $F$ if the geometric intersection number of any pair of elements in the system is exactly $k$. We let $\widehat{\mathscr{L}}(F,k)$ or $\widehat{\mathscr{A}}(F,k)$ be the collection of all complete $k$-systems of loops or of arcs on $F$.

\end{defi}

\begin{defi}[Equivalent systems]
\label{defi: Equivalent systems}

Let $\Omega = \{\beta_i\}_{i \in \mathscr{I}}, \bar{\Omega} = \{ \bar{\beta}_i \}_{i \in \bar{\mathscr{I}}}$ be respective systems of curves on $F$. We say $\Omega$ and $\bar{\Omega}$ are \emph{equivalent up to mapping class group action} if there is a homeomorphism $\varphi : F \rightarrow F$ and a bijection $\psi : \mathscr{I} \rightarrow \bar{\mathscr{I}}$ such that $\varphi \circ \beta_i$ is isotopic to $\bar{\beta}_{\psi(i)}$ for every $i \in \mathscr{I}$. We call the pair $(\varphi, \psi)$ an \emph{equivalence} between $\Omega$ and $\bar{\Omega}$. In particular, we say $\Omega$ and $\bar{\Omega}$ are \emph{isotopic} if there is a bijection $\psi : \mathscr{I} \rightarrow \bar{\mathscr{I}}$ such that $\beta_i$ is isotopic to $\bar{\beta}_{\psi(i)}$ for every $i \in \mathscr{I}$.

%We say $\Omega$ and $\widehat{\Omega}$ are \emph{regular neighbourhood equivalent} if there is a homeomorphism $\varphi : N \left(\left\{  \beta_i \middle| i \in \mathscr{I} \right\} \right) \rightarrow N \left( \left\{  \widehat{\beta_i} \middle| i \in \widehat{\mathscr{I}} \right\} \right)$ and a bijection $\psi : \mathscr{I} \rightarrow \widehat{\mathscr{I}}$ such that for every $i \in \mathscr{I}$, we have that $\varphi \circ \beta_i$ is isotopic to $\widehat{\beta_{\psi(i)}}$ on $N \left( \left\{  \widehat{\beta_i} \middle| i \in \widehat{\mathscr{I}} \right\} \right)$.

\end{defi}

%\begin{rmk}

%The system homeomorphism and the system neighbourhood homeomorphism cannot be deduced from each other.

%\end{rmk}

\cref{SystemEquivalence} is an example of two equivalent complete $1$-systems of loops on $S_{2,0}$.

\begin{figure}[H]
\centering
\includegraphics[width=1.0\textwidth]{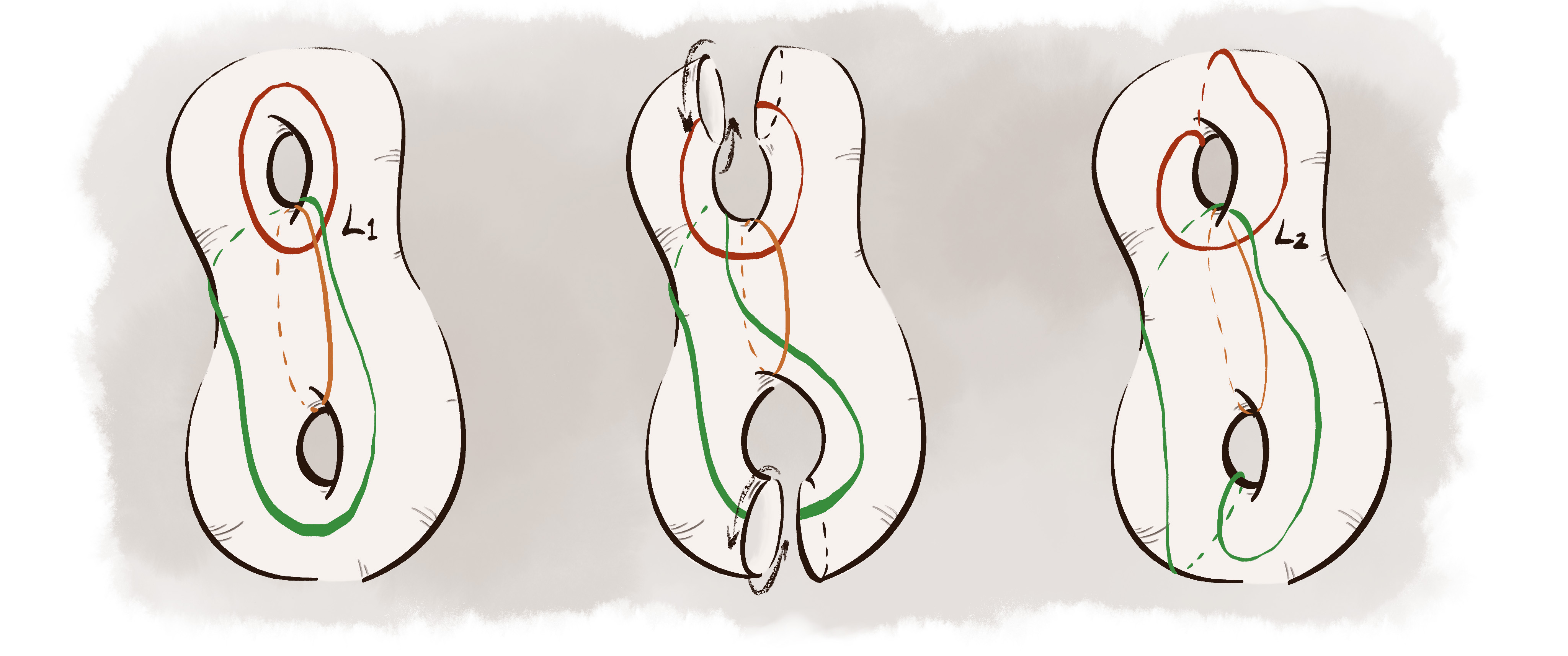}
\caption{$L_1$ and $L_2$ are two equivalent complete $1$-systems of loops on $S_{2,0}$.}
\label{SystemEquivalence}
\end{figure}

\section{Mapping Class Groups of Punctured Projective Planes}\label{section3}

Recall that the mapping class group of a surface $F$ (possibly with punctures and boundary components) is defined to be
the group of the isotopy classes of all diffeomorphisms $\varphi : F \rightarrow F$ which map $\partial F$ to itself \cite{MR1890954}. Now we discuss the mapping class groups of punctured projective planes.

\begin{defi}[Half twists]
\label{Defi: HalfTwists}
Consider a surface $F$ with at least two punctures and a subsurface $F'$ embedded in $F$ which is topologically a disk with two punctures $P_1,P_2$. Let $\alpha$ and $\beta$ be disjoint two simple arcs in $F'$ and linking $P_1$ and $P_2$. By interchanging $P_1$ and $P_2$ along $\alpha$ and $\beta$, we get a diffeomorphism of $F'$ fixing a neighborhood of $\partial F'$. this diffeomorphism can be extended to $F$ by identity. This diffeomorphism or its isotopy class is called a \emph{half twist} of $P_1$ and $P_2$ (see \cref{HalfTwist}).
\end{defi}

\begin{figure}[H]
\centering
\includegraphics[width=1.0\textwidth]{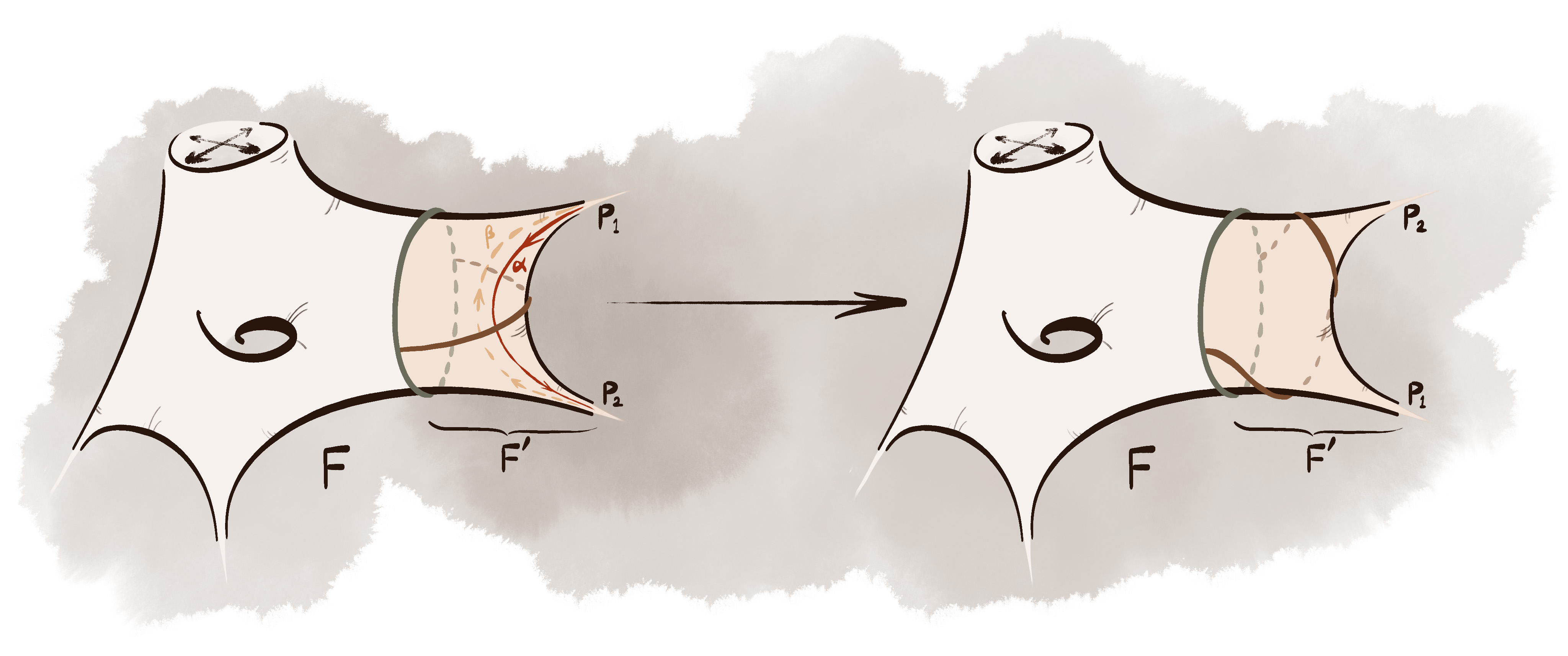}
\caption{This is an example of a half twist between $P_1$ and $P_2$, where they exchange in $F'$. The change in the brown line visually represents the exchange of the two points.}
\label{HalfTwist}
\end{figure}

\begin{defi}[Puncture slides]
\label{Defi: PunctureSlides}
Consider a non-orientable surface $F$ with punctures and a subsurface $F'$ embedded in $F$ which is topologically a Möbius band with one puncture $P$. By sliding $P$ along the core of the Möbius band, we get a diffeomorphism of $F'$ fixing a neighborhood of $\partial F'$. this diffeomorphism can be extended to $F$ by identity. This diffeomorphism or its isotopy class is called a \emph{puncture slide} of $P$ \cite{MR1890954} (see \cref{PunctureSlide}).
\end{defi}

\begin{figure}[H]
\centering
\includegraphics[width=1.0\textwidth]{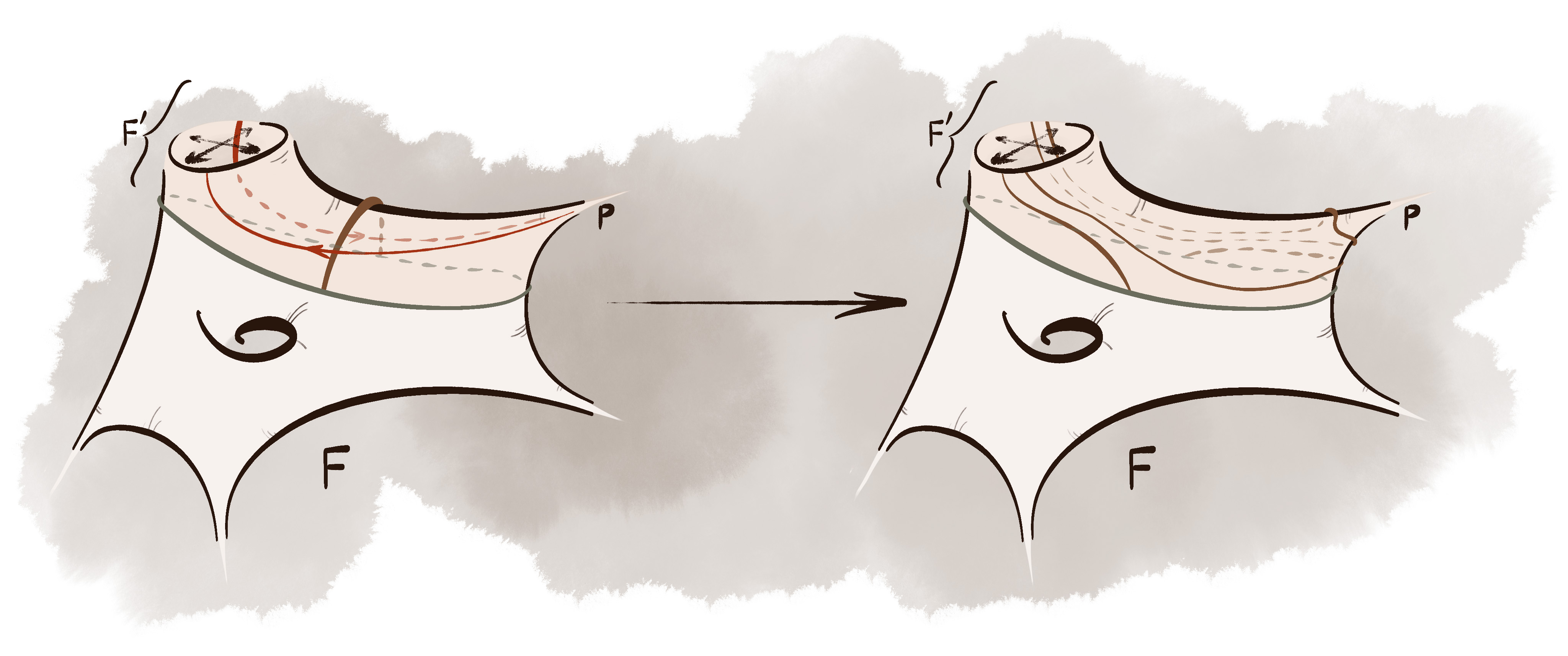}
\caption{This is an example of a puncture slide for $P$. In $F'$, $P$ moves along the red line, passing through the cross cap and returning to its original position. The change in the brown line represents this puncture slide.}
\label{PunctureSlide}
\end{figure}

\begin{thm}[Generators of mapping class groups of punctured disks]
\label{punctured disk MCG generators}
Let $D_n$ be a closed disk with $n$ punctures $P_1, \cdots, P_n$. Then the mapping class group of $D_n$ is generated by the set of all half-twists.
\end{thm}

\begin{thm}[Generators of mapping class groups of punctured projective planes {(\cite[Theorem 4.4]{MR1890954})}]
\label{MCG generators}
Let $N_{1,n}$ be a projective plane with $n$ punctures $P_1, \cdots, P_n$. Then the mapping class group of $N_{1,n}$ is generated by all half twists and all puncture slides.
\end{thm}

\section{Equivalence of Maximal $1$-Systems and Complete Maximal $1$-Systems}\label{section4}

\begin{pro}
\label{2-sided separating}
Let $\gamma$ be a $2$-sided loop on $N_{1,n}$, then $\gamma$ is separating.

\begin{proof}
This is \cite[Lemma 5.3]{chen2024systems}. For the completeness, we give the proof here.

Assume that $\gamma$ is a non-separating $2$-sided loop on $F$. Its regular neighborhood $W(\gamma)$ is an annulus. Since $\gamma$ is non-separating, $F \setminus W(\gamma)$ is connected. Therefore, we can find a simple arc $\alpha$ in $F \setminus W(\gamma)$, connecting the two boundaries of $W(\gamma)$. We can determine the topological type of $F' := W(\gamma) \cup W(\alpha)$ by the Euler characteristic (it is $-1$) and the number of boundary components (it is $1$), which is either $S_{1,1}$ or $N_{2,1}$ depending on whether the band $W(\alpha)$ connects the boundaries of $W(\gamma)$ in an orientation-preserving or reversing manner. If $F' = N_{2,1}$, since $F'$ is a subsurface of $F$ and already has two cross-caps, the number of cross-caps of $F$ must be at least $2$. This contradicts the fact that $F$ is $N_{1,n}$. If $F'= S_{1,1}$, then $F \setminus F'$ is non-orientable, otherwise, if both $F'$ and its complement were orientable, their connected sum, $F$, would also be orientable, leading to a contradiction. Thus, the surface $F \setminus F'$ must be $N_{d,n+1}$ for $d \geq 1$. Therefore, $F$ has at least one annulus (in $F'$) and one cross-cap (in the complement of $F'$). Hence, $F$ must have at least three cross-caps, which contradicts $F = N_{1,n}$.
\end{proof}
\end{pro}

\begin{pro}
\label{separating cross even number}
Let $\gamma_1$ be a separating loop on a surface $F$. Then, any loop $\gamma_2$ on $F$ that is in minimal position with respect to $\gamma_1$ intersects $\gamma_1$ an even number of times.

\begin{proof}
Since $\gamma_1$ is separating, $F \setminus \gamma_1$ has two components $F_1$ and $F_2$. Consider $\gamma_2 \cap (F_1 \cup \gamma_1)$, it is a collection of arcs whose endpoints are on $\gamma_1$. There is an even number of endpoints, since every arc has two endpoints. These endpoints are precisely all of the intersection points between $\gamma_1$ and $\gamma_2$.
\end{proof}
\end{pro}

\begin{cor}
\label{separating disjoint}
Let $L$ be a $1$-system of loops on $N_{1,n}$, and $\gamma_1 \in L$ be a $2$-sided loop, then $\gamma_1$ is disjoint from all other loops in $L$.

\begin{proof}
By \cref{2-sided separating}, $\gamma_1$ is separating. By \cref{separating cross even number}, the intersection number between $\gamma_1$ and any other loop $\gamma_2$ in $L$ is even. Since $L$ is a $1$-system, $i(\gamma_1, \gamma_2) \leq 1$. Hence, $i(\gamma_1, \gamma_2) = 0$, and the two loops $\gamma_1$ and $\gamma_2$ are disjoint.
\end{proof}
\end{cor}

\begin{pro}
\label{1-sided loops intersect}
Let $\gamma_1, \gamma_2$ be two $1$-sided loops on $N_{1,n}$, then $\gamma_1$ and $\gamma_2$ intersect.

\begin{proof}
Assume that $\gamma_1, \gamma_2 \in L$ are disjoint. We cut $N_{1,n}$ along $\gamma_1$ and obtain an orientable surface $S_{0,n+1}$. Since $\gamma_2$ does not intersect $\gamma_1$, $\gamma_2$ remains a loop on $S_{0,n+1}$. This leads to a contradiction: $\gamma_2$ is $1$-sided, but it lies on an orientable surface.
\end{proof}
\end{pro}

\begin{lem}
\label{lem: max no 2-sided}
Let $L$ be a maximal $1$-system of loops on $F := N_{1,n}$, then all loops $\gamma$ in $L$ are $1$-sided.

\begin{proof}
Since $L$ is maximal, by \cite[Corollary 1.6]{chen2024systems},
\begin{align}
\label{ineq1}
\# L = \|\mathscr{L}(F,1)\|_\infty \geq \|\widehat{\mathscr{L}}(F,1)\|_\infty = \tfrac{1}{2}n^2 - \tfrac{1}{2}n + 1.
\end{align}

We let $L_1$ be the collection of $1$-sided loops in $L$, and let $L_2$ be the collection of $2$-sided loops in $L$. Assume that there are $m\geq 1$ loops in $L$ which are $2$-sided, that is $\# L_2 = m$. By \cref{2-sided separating} and \cref{separating disjoint}, $L_2$ is a $0$-system and $L_2$ separates $F$ into $m + 1$ pieces, which we labeled by $F_0 , \cdots, F_m$. Exactly one of the $\{F_{i}\}_i$ is non-orientable (see \cref{AllLoopsAreOneSided}), without loss of generality, let $F_0$ be the non-orientable subsurface. We iteratively cut the surface along the $2$-sided loops in $L_2$, starting from those loops that are farthest from the non-orientable part. The orientable subsurface that is cut off each time we cut along a loop must be a sphere with at least three combined boundary components and punctures, meaning that the non-orientable surface loses at least one boundary component after each cutting. Hence, suppose that $F_0 = N_{1,t}$, then $t \leq n-m$. Note that $t$ cannot be $1$, because if it were equal to $1$, the $2$-sided loop that is the unique boundary of $F_0$ would no longer be essential. Hence, we have $n-m \geq 2$.

\begin{figure}[H]
\centering
\includegraphics[width=1.0\textwidth]{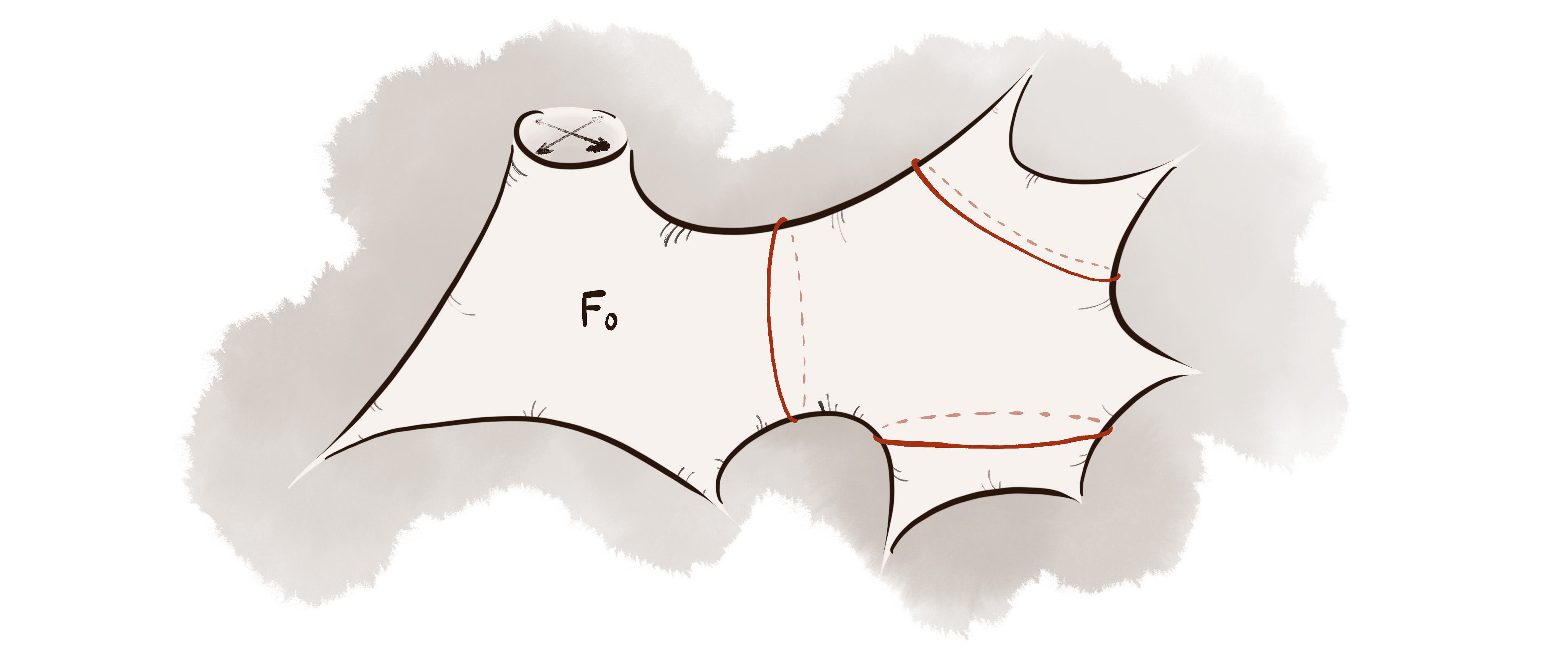}
\caption{$L_2$ divides $F$ into $m+1$ pieces. Assume that $F_0$ is the non-orientable piece among them.}
\label{AllLoopsAreOneSided}
\end{figure}

Since the regular neighborhoods of $1$-sided loops are Möbius strips, all $1$-sided loops (i.e., the loops in $L_1$) lie entirely within $F_0$. By \cref{1-sided loops intersect}, $L_1$ is complete. Hence, by \cite[Corollary 1.6]{chen2024systems},
\begin{align}
\# L &= \# L_1 + \# L_2 \leq \|\widehat{\mathscr{L}}(F_0,1)\|_\infty + m \leq \|\widehat{\mathscr{L}}(N_{1,n-m},1)\|_\infty + m \notag \\ 
&= \tfrac{1}{2}(n-m)^2 - \tfrac{1}{2}(n-m) + 1 + m = \tfrac{1}{2}n^2 - \tfrac{1}{2}n + 1 + \tfrac{1}{2} m^2 + \tfrac{3}{2} m - mn \notag \\
&\leq \tfrac{1}{2}n^2 - \tfrac{1}{2}n + 1 + \tfrac{1}{2} m^2 + \tfrac{3}{2} m - m(m+2) = \tfrac{1}{2}n^2 - \tfrac{1}{2}n + 1 + \tfrac{1}{2} (-m^2 - 3m) \notag \\
\label{ineq2}
&\leq \tfrac{1}{2}n^2 - \tfrac{1}{2}n + 1.
\end{align}

Combining \cref{ineq1} and \cref{ineq2}, $\# L = \tfrac{1}{2}n^2 - \tfrac{1}{2}n + 1$. At this point, the inequalities in the above equation are inequalities, i.e., $m=0$. Hence, all loops $\gamma$ in $L$ are $1$-sided.
\end{proof}
\end{lem}

\begin{lem}
\label{lem: max to complete}
Let $L$ be an maximal $1$-system of loops on $N_{1,n}$, then $L$ is complete (see \cref{defi: Complete $k$-systems of curves}).

\begin{proof}
Assume that $L$ is maximal but incomplete. Let $L_0 \subset L$ be a largest $0$-system in $L$. Since $L$ is incomplete, $\# L_0 \geq 2$. Hence, by \cref{1-sided loops intersect}, there is a $2$-sided loop $\gamma$ in $L_0 \subset L$, \cref{lem: max no 2-sided} then contradicts the maximality of $L$.
\end{proof}
\end{lem}

% \begin{thm}
%     Let $L$ be a $1$-system of loops on $N_{1,n}$, then $L$ is maximal if and only if $L$ is maximal and complete.

%     \begin{proof}
%         If $L$ is maximal and complete, then $L$ is maximal. If $L$ is maximal, by \cref{lem: max to complete}, $L$ is complete, hence, $L$ is maximal and complete.
%     \end{proof}
% \end{thm}

\section{Watermelons and Standard Maximal Watermelons}\label{section5}

To study the uniqueness of maximal complete $1$-systems of loops, we introduce a concept called \emph{watermelons}, which graphically represent loops on the punctured projective plane.

\subsection{Watermelons and marked complete 1-systems of loops}
\label{subsection: Watermelons}

In \cref{subsection: Watermelons}, we will first define what marked complete $1$-systems of loops and watermelons are. The ultimate goal is to establish a one-to-one correspondence (with the relevant definitions provided immediately following the theorem):

\begin{thm}[$\mathcal{M}=\mathcal{W}$]
\label{thm: watermelons and marked loops systems}
On $N_{1,n}$, there exists a one-to-one correspondence between:
\begin{align}
\mathcal{M} := \{ \text{Marked complete $1$-syste} & \text{ms of loops}\} / \text{mark equivalence} \notag \\
&\rotatebox{90}{$\longleftrightarrow$} {^{^{1:1}}}  \notag \\
\mathcal{W} := \{ \text{Watermelons} \} / \text{w} & \text{atermelon equivalence}. \notag
\end{align}

\end{thm}

\begin{defi}[Marked complete $1$-systems of loops]
\label{defi: Marked complete 1-systems of loops}
Let $L$ be a complete $1$-system of loops on $N_{1,n}$, and let $\gamma_0$ be an \emph{oriented} loop in $L$. We call $(L, \gamma_0)$ a \emph{marked complete $1$-system of loops}, and refer to $\gamma_0$ as the \emph{mark} of $(L, \gamma_0)$.
\end{defi}

\begin{defi}[Mark equivalence]
\label{defi: Mark equivalence}
Let $(L, \overrightarrow{\gamma_0})$ and $(L', \overrightarrow{\gamma'_0})$ be two marked complete $1$-systems of loops. We say that $(L, \overrightarrow{\gamma_0})$ and $(L', \overrightarrow{\gamma'_0})$ are \emph{mark equivalent} if there exists an equivalence $(\varphi, \psi)$ between $L$ and $L'$ as systems of loops (see \cref{defi: Equivalent systems}), such that $\psi(0) = 0$ and the loops $\varphi \circ \gamma_0$ and $\gamma'_{\psi(0)} = \gamma'_0$ are isotopic and compatibly oriented.
\end{defi}

Let $D_n \subset \mathbb{C}$ denote the unit disk with $n$ interior punctures; identifying antipodal points on $\partial D_n$, i.e., $x \sim -x$, gives a representation of $N_{1,n}$ as $D_n / {\sim}$.

\begin{defi}[Watermelons (see \cref{WatermelonDiagram})]
We say a $1$-system of arcs $A$ (see \cref{defi: Systems of curves}, which requires that all arcs be in pairwise minimal position) on $D_n$ (where $n \geq 2$) is a \emph{watermelon}.
\end{defi}

\begin{figure}[H]
\centering
\includegraphics[width=1.0\textwidth]{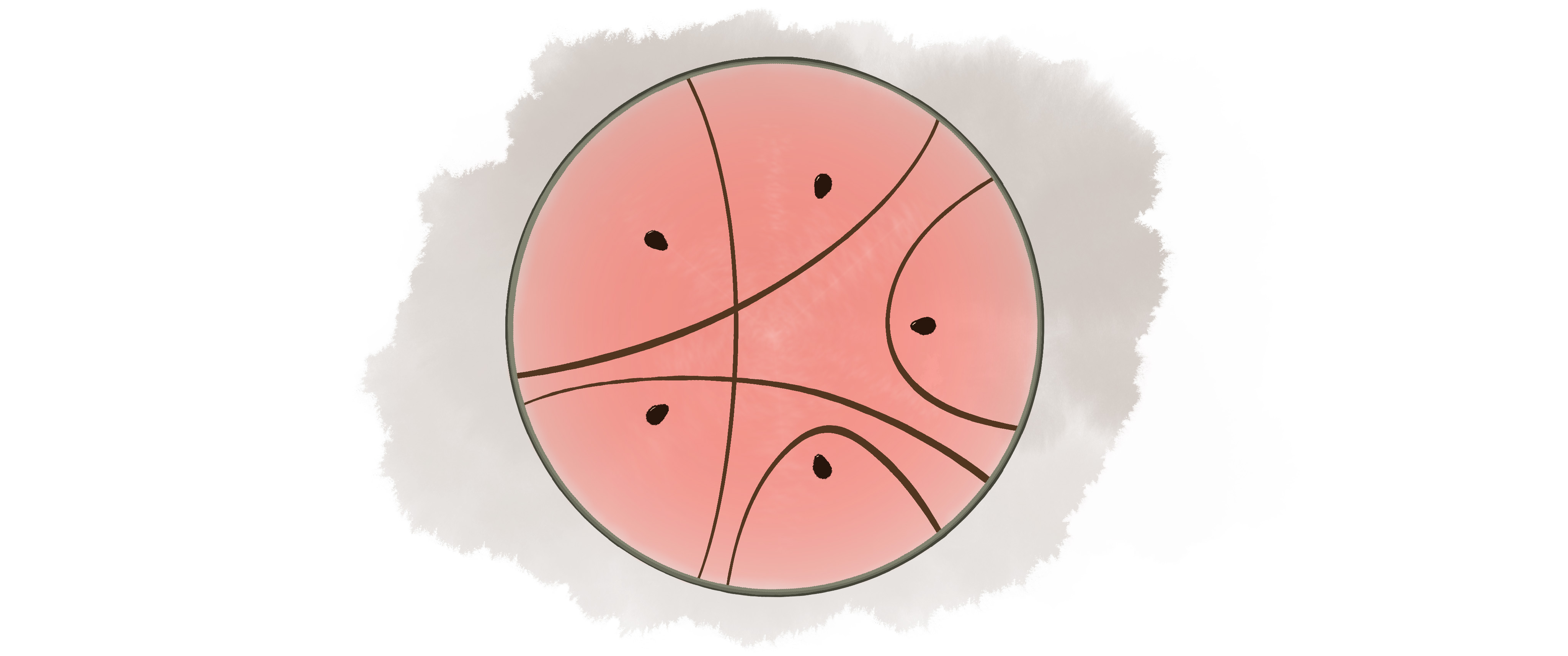}
\caption{An example of a watermelon.}
\label{WatermelonDiagram}
\end{figure}

\begin{rmk}
Regard $D_n$ as a cross-section of a watermelon: we may think of $\partial D_n$ as the rind, and the $n$ punctures as the seeds.
\end{rmk}

\begin{defi}[Watermelon equivalence]
\label{defi: Watermelon equivalence}
Let $A = \{\alpha_i \mid i \in \mathscr{I}\}$ and $A' = \{\alpha'_i \mid i \in \mathscr{I}' \}$ be two watermelons. We say that $A$ and $A'$ are \emph{watermelon equivalent} if there exists an orientation-preserving self-homeomorphism $\varphi$ (possibly permuting punctures) of $D_n$ and a bijection $\psi : \mathscr{I} \rightarrow \mathscr{I}'$ such that $\varphi \circ \alpha_i$ is isotopic to $\alpha'_{\psi(i)}$ for every $i \in \mathscr{I}$.
\end{defi} 

\begin{rmk}
The notion of watermelon equivalence differs from that of equivalence as systems of arcs (\cref{defi: Equivalent systems}): the latter allows for identification which inverse surface oritations.
\end{rmk}

To prove \cref{thm: watermelons and marked loops systems}, we define two maps $f_1: \mathcal{M} \rightarrow \mathcal{W}$ and $f_2: \mathcal{W} \rightarrow \mathcal{M}$, show that $f_1$ and $f_2$ are inverse each other. We now construct the map $f_1 : \mathcal{M} \rightarrow \mathcal{W}$.

\begin{defi}[Associated watermelons]
\label{Associated watermelons}

Let $L = \{ \gamma_0, \cdots, \gamma_{m-1} \}$ be a complete $1$-system of loops on $N_{1,n}   \cong \partial D_n / {\sim}$, consisting of $m \geq 2$ loops. Choose an orientation on $\gamma_0 \in L$, and consider the marked complete $1$-system of loops $(L, \overrightarrow{\gamma_0})$. Then there exists a self-homeomorphism $\varphi$ of $D_n / {\sim}$ such that $\varphi(\overrightarrow{\gamma_0})$ is $\partial D_n / {\sim}$ endowed with the counterclockwise orientation. For convenience, all loops are henceforth understood to be their images under $\varphi$. Since each $\gamma_i$ (for $i \geq 1$) intersects $\gamma_0$ exactly once, the loop $\gamma_i$ can be obtained by identifying the antipodal endpoints of an arc. We denote this arc in $D_n$ by $\alpha^{\dagger}(\gamma_i)$ for $i = 1, \cdots, m-1$, and refer to $\alpha^{\dagger}(\gamma_i)$ as the \emph{cut-open arc} of $\gamma_i$. We respectively choose a representative $\alpha(\gamma_i)$ for each isotopy class $[\alpha^{\dagger}(\gamma_i)]$ so that each pair $\alpha(\gamma_i)$ and $\alpha(\gamma_j)$ intersects transversely and is in minimal position. We refer to each such arc $\alpha(\gamma_i)$ as an \emph{associated arc} of $\gamma_i$. We define the set $f_1\left((L,\overrightarrow{\gamma_0})\right) := \{ \alpha(\gamma_i) \mid i = 1, \cdots, m-1 \} =: A_L$ to be the \emph{associated watermelon} of $L$.

\end{defi}

\begin{rmk}
We prove that, up to watermelon equivalence, the definition of $f_1$ is independent of the choice of $\varphi$ and of the choice of representative of the marked system of loops within its mark equivalence class. It follows that $f_1$ is a well-defined map from $\mathcal{M}$ to $\mathcal{W}$.

Let $(L, \overrightarrow{\gamma_0}) = (\{\gamma_i\}_{i=0}^{l}, \overrightarrow{\gamma_0})$ and $(L', \overrightarrow{\gamma'_0}) = (\{\gamma'_i\}_{i=0}^{l}, \overrightarrow{\gamma'_0})$ be two marked complete $1$-systems of loops that are mark equivalent, and let $(\Phi,\Psi) : (L,\overrightarrow{\gamma_0}) \rightarrow (L',\overrightarrow{\gamma'_0})$ be a mark equivalence between them (see \cref{defi: Mark equivalence} and \cref{defi: Equivalent systems}). Suppose $\varphi$ is a self-homeomorphism of $D_n / {\sim}$ that sends $\overrightarrow{\gamma_0}$ to the $\partial D_n / {\sim}$, endowed with the counterclockwise orientation. Similarly, let $\varphi'$ be a self-homeomorphism of $D_n / {\sim}$ that sends $\overrightarrow{\gamma'_0}$ to $\partial D_n / {\sim}$, endowed with the counterclockwise orientation. Then $m := \varphi' \circ \Phi \circ \varphi^{-1}$ is a self-homeomorphism of $D_n / {\sim}$ that sends $\overrightarrow{\gamma_0}$ to $\overrightarrow{\gamma'_0}$ in an orientation-preserving way. We may thus regard $m$ as an orientation-preserving self-homeomorphism of $D_n$ that restricts to the identity on $\partial D_n$ up to homotopy. Moreover, for each $i = 0, \cdots, l$, $m(\gamma_i)$ is isotopic to $\gamma_{\Psi(i)}$ (as a loop). Consequently, $m(\alpha^\dagger(\gamma_i))$ is isotopic to $\alpha^\dagger(\gamma_{\Psi(i)})$ (as an arc), and thus $m(\alpha(\gamma_i))$ is isotopic to $\alpha(\gamma_{\Psi(i)})$. Therefore, $f_1((L, \overrightarrow{\gamma_0}))$ and $f_1((L', \overrightarrow{\gamma'_0}))$ are watermelon equivalent.
\end{rmk}

\begin{defi}[Associated systems of loops]
\label{Associated systems of loops}
We now construct the inverse map $f_2: \mathcal{W} \rightarrow \mathcal{M}$ for \cref{thm: watermelons and marked loops systems}. We first isotope each arc in $A$ to a smooth representative. Then, we may choose a sufficiently small positive number $\epsilon > 0$ such that the following conditions are satisfied (see the left picture in \cref{AssociatedSystem}):

\begin{itemize}
\item let $R_\epsilon := \{ x \in D_n \mid d(x,\partial D_n) \leq \epsilon \}$ be closed annulus in $D_n$ that contains none of the punctures, here $d$ denotes Euclidean distance;
\item for every $\alpha \in A$, $\alpha \cap (D_n \setminus R_\epsilon)$ is a connected arc;
\item for each pair of arcs $\alpha_1, \alpha_2 \in A$, we have $\alpha_1 \cap \alpha_2 \cap R_\epsilon$ is empty.
\end{itemize}

\begin{figure}[H]
\centering
\includegraphics[width=1.0\textwidth]{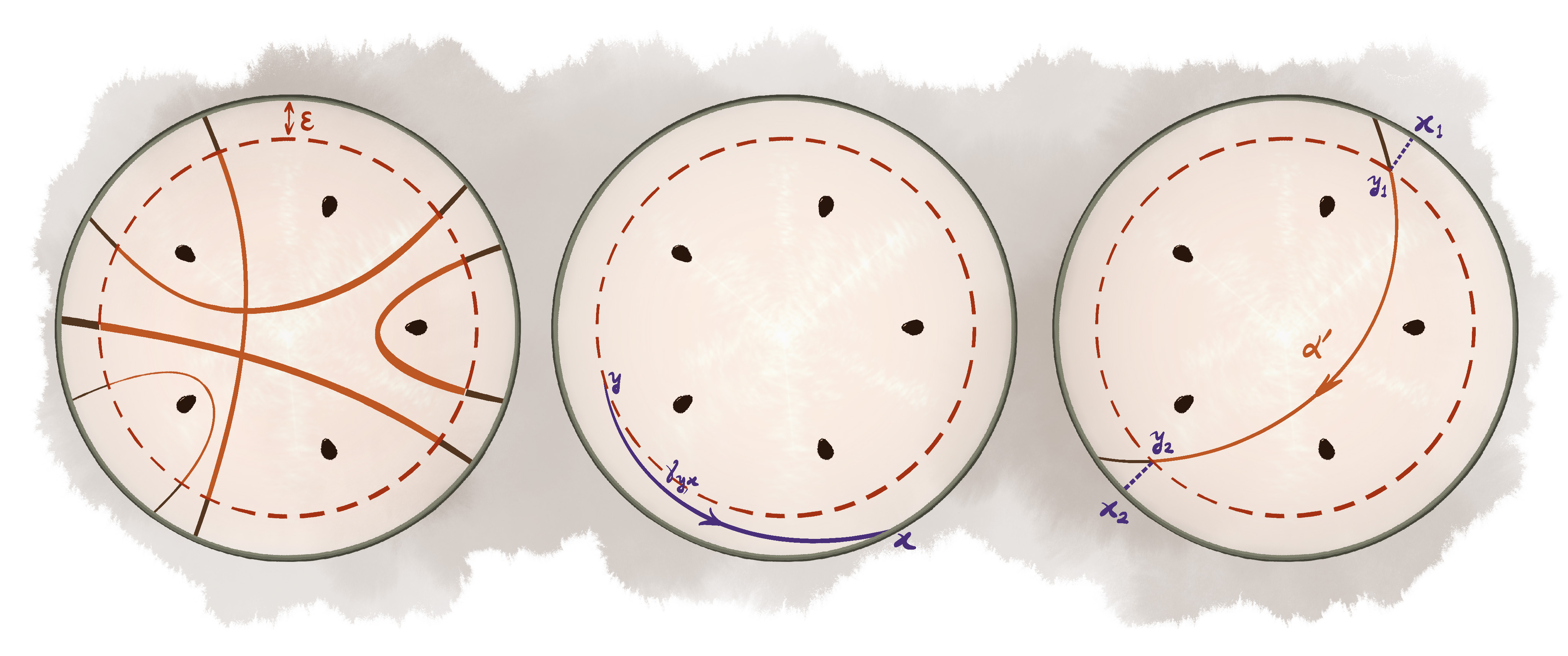}
\caption{Some preparatory work for the construction of the associated system of a watermelon.}
\label{AssociatedSystem}
\end{figure}

We orient $\partial D_n$ in the counterclockwise direction. For a point $y = (1-\epsilon) e^{\theta_y \sqrt{-1}}$ on $C_\epsilon := \{ x \in D_n \mid d(x,\partial D_n) = \epsilon \}$ and a point $x = e^{\theta_x \sqrt{-1}}$ on $\partial D_n$. Without loss of generality, we suppose $0 \leq \theta_x - \theta_y < 2 \pi$. We define
\begin{align}
\alpha_{y,x}:[0,1] & \rightarrow R_\epsilon,  \notag \\
t & \mapsto ((1-t)(1-\epsilon)+t)e^{\left(\theta_y + t (\theta_x-\theta_y)\right) \sqrt{-1}}, \notag
\end{align}
which is an arc in $R_\epsilon$ from $y$ to $x$ that spirals outward in the counterclockwise direction (see the middle picture in \cref{AssociatedSystem}).

Next, we construct the loops in $f_2(A) =: (L_A, \overrightarrow{\gamma_0})$, which we call the \emph{associated marked complete $1$-system of loops} of the watermelon $A$. For each arc $\alpha \in A$ (note that $\alpha$ is an injection from $[0,1]$ into $D_n$), let $y_1$ be the first point where $\alpha$ intersects the circle $C_\epsilon$, and $y_2$ the second point where $\alpha$ intersects $C_\epsilon$. Let $t_1 := \alpha^{-1}(y_1) < t_2 := \alpha^{-1}(y_2)$, and denote by $\alpha' := \alpha|_{[t_1, t_2]}$ the sub-arc between these points. Let $x_i := \dfrac{y_i}{ |y_i| }$ for $i = 1, 2$; these are points on the boundary of $D_n$ (see the right picture in \cref{AssociatedSystem}). We define a loop $\gamma_\epsilon(\alpha)$ in $D_n / {\sim}$, which we typically abbreviate as $\gamma(\alpha)$. This loop, called the \emph{associated loop} of $\alpha$, is defined as follows:
\begin{align}
\label{eq: gamma_alpha}
\gamma(\alpha) =
\begin{cases}
{\alpha_{y_1, -x_2}}^\# * \alpha' * \alpha_{y_2, x_2}, \quad &\text{ if } \theta_{y_1} - \theta_{y_2} \in [0,\pi] \mod 2\pi,\\
{\alpha_{y_2, -x_1}}^\# * {\alpha'}^\# * \alpha_{y_1, x_1}, \quad &\text{ if } \theta_{y_1} - \theta_{y_2} \in (\pi,2\pi) \mod 2\pi\\
\end{cases}
\end{align}
where:
\begin{itemize}
\item the symbol $*$ denotes the concatenation of oriented paths.
\item for any path $\alpha$, the notation $\alpha^\#$ denotes the inverse path (i.e., $\alpha$ traversed in the reverse direction).
\item the modular intervals used in the case distinction are defined by:
\begin{align}
[0, \pi] \mod 2\pi &:= \bigcup_{n \in \mathbb{Z}} [0 + 2n\pi, \pi + 2n\pi],  \notag \\
(\pi, 2\pi) \mod 2\pi &:= \bigcup_{n \in \mathbb{Z}} (\pi + 2n\pi, 2\pi + 2n\pi). \notag
\end{align}
\end{itemize}
See \cref{GammaAlpha} for an illustration.

\begin{figure}[H]
\centering
\includegraphics[width=1.0\textwidth]{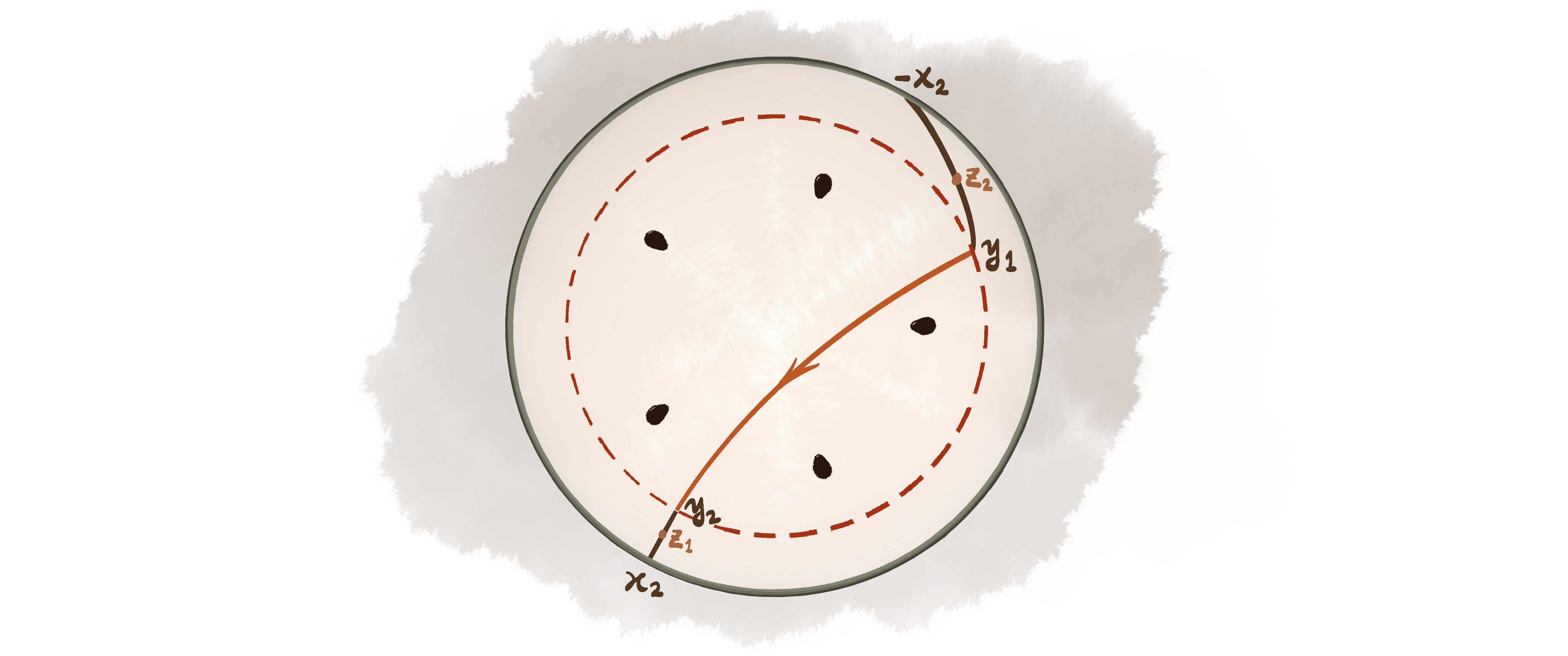}
\caption{The specific construction of $\gamma_\alpha$.}
\label{GammaAlpha}
\end{figure}

We define $L_A := \{\gamma(\alpha) \mid \alpha \in A \} \cup \{\gamma_0\}$, where $\gamma_0 = \partial D_n / {\sim}$, and regard $\overrightarrow{\gamma_0}$, equipped with the counterclockwise orientation, as the mark of $L_A$. Next, we prove that $L_A$ is a complete $1$-system of loops; that is, every loop in $L_A$ is simple, and each pair of loops intersects exactly once.

We already know that $\gamma_0$ is simple and each loop intersects $\gamma_0$ exactly once.

For each $\gamma(\alpha) \in L_{A}$, to prove that $\gamma(\alpha)$ is simple, it suffices to show that ${\alpha_{y_1, -x_2}}^\#$ and $\alpha_{y_2, x_2}$ are disjoint. First, consider the case where $\theta_{y_1} - \theta_{y_2} \in [0, \pi] \mod 2\pi$. Since for every $z_1 \in \alpha_{y_2, x_2}$ we have $\theta_{z_1} = \theta_{y_2}$ (see \cref{GammaAlpha}), to prove $z_2 \neq z_1$ for every $z_2 \in {\alpha_{y_1, -x_2}}^\#$, it suffices to show $\theta_{z_2} \neq \theta_{y_2}$. This follows from the assumptions that $\theta_{y_2} - \theta_{y_1} \in (\pi, 2\pi) \mod 2\pi$ and $\theta_{z_2} - \theta_{y_1} \leq \theta_{-x_2} - \theta_{y_1} \in [0, \pi) \mod 2\pi$. The proof for the case $\theta_{y_1} - \theta_{y_2} \in (\pi, 2\pi) \mod 2\pi$ is similar.

For each pair of loops $\gamma(\alpha_1), \gamma(\alpha_2) \in L_{A}$, we aim to prove that they intersect exactly once. Observe that the cut-open arc $\alpha^\dagger(\gamma(\alpha_1))$ of $\gamma(\alpha_1)$ (see \cref{Associated watermelons}) separates $D_n$ into two connected components. Since $\alpha^\dagger(\gamma(\alpha_2))$ is an arc with endpoints lying in different components, it must intersect $\alpha^\dagger(\gamma(\alpha_1))$ an odd number of times.

We separately consider the number of intersection points of $\gamma(\alpha_1), \gamma(\alpha_2)$ in $R_\epsilon$ and $D_n \setminus R_\epsilon$.

For the number of intersection in $D_n \setminus R_\epsilon$, since $\alpha_1$ and $\alpha_2$ intersect at most once, their restrictions $\alpha'_1 := \alpha_1 \cap (D_n \setminus R_\epsilon)$ and $\alpha'_2 := \alpha_2 \cap (D_n \setminus R_\epsilon)$ also intersect at most once.

For the number of intersections in $R_\epsilon$, we aim to prove that $\gamma(\alpha_1) \cap R_\epsilon$ and $\gamma(\alpha_2) \cap R_\epsilon$ intersect at most once. We proceed by contradiction: if not, the intersection number between $\gamma(\alpha_1) \cap R_\epsilon$ and $\gamma(\alpha_3) \cap R_\epsilon$ is at least two. We write $\gamma(\alpha_1) \cap R_\epsilon = l_1 \cup c_1$ and $\gamma(\alpha_2) \cap R_\epsilon = l_2 \cup c_2$, where $l_i$ denotes the connected component of $\gamma(\alpha_i) \cap R_\epsilon$ that is a straight-line segment, and $c_i$ denotes the spiraling segment. Since $l_1$ and $l_2$ are disjoint, and $c_i$ are linearly increasing functions of the argument with angular increments less than $\pi$, it follows that $c_1$ and $c_2$ cannot intersect twice. Moreover, if $c_1$ intersects $l_2$ (or $c_2$ intersects $l_1$), then $c_1$ and $c_2$ do not intersect. Since $\gamma(\alpha_1) \cap R_\epsilon$ and $\gamma(\alpha_2) \cap R_\epsilon$ intersect at least twice, we have $l_1$ and $c_2$ intersecting (once) and $l_2$ and $c_1$ intersecting (once), which leads to a contradiction that both the starting and ending points of $\alpha(\gamma(\alpha_2))$ lie on the same side of $\alpha(\gamma(\alpha_1))$ (see \cref{ExactlyOnce}).

\begin{figure}[H]
\centering
\includegraphics[width=1.0\textwidth]{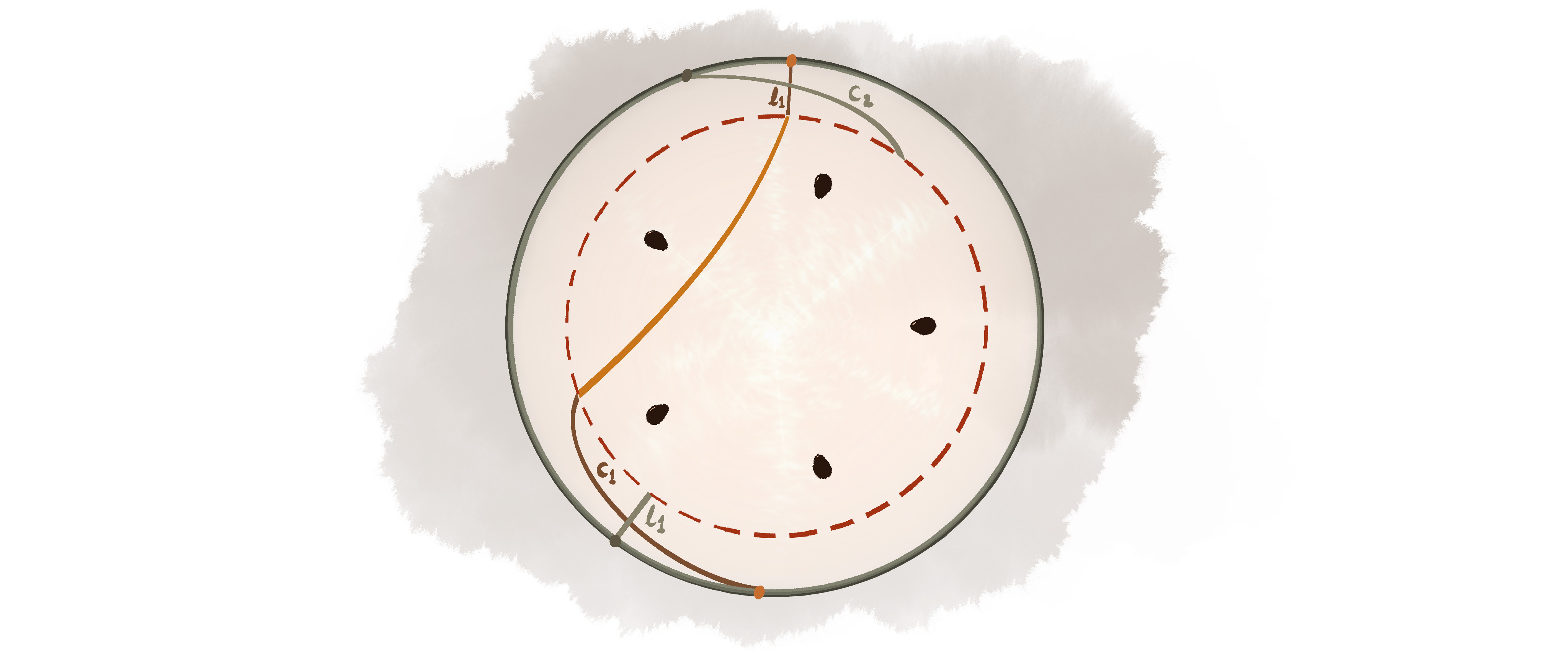}
\caption{The elements in $L_A$ intersect pairwise exactly once.}
\label{ExactlyOnce}
\end{figure}

Hence $i(\gamma(\alpha_1),\gamma(\alpha_2))$ is at most two. Combining the fact that they intersect an odd number times, the intersection number is exactly one, as desired.

\end{defi}

\begin{rmk}
We prove that the definition of $f_2$ does not depend on the choice of $\epsilon$ satisfying the required conditions (and the variations in the definition of $\gamma_\epsilon(\alpha) = \gamma(\alpha)$ in \cref{eq: gamma_alpha} arising from different choices of $\epsilon$ do not affect the outcome), nor on the choice of representative of the watermelon within its equivalence class. It follows that $f_2$ is a well-defined map from $\mathcal{W}$ to $\mathcal{M}$.

Let $A = \{\alpha_i\}_{i=0}^{l}$ and $A' = \{\alpha'_i\}_{i=0}^{l}$ be two watermelon-equivalent watermelons, and let $(\Phi,\Psi) : A \rightarrow A'$ be a watermelon equivalence between them. Intuitively, the loop $\gamma(\alpha)$ is formed by moving the endpoints of $\alpha$ to antipodal points and identifying them. Hence, $\alpha$ and the cut-open arc $\alpha^\dagger(\gamma(\alpha))$ are isotopic as arcs. Since $\Phi(\alpha_i)$ and $\alpha'_{\Psi(i)}$ are isotopic (as arcs), it follows that $\Phi(\alpha^\dagger(\gamma(\alpha_i)))$ and $\alpha^\dagger(\gamma(\alpha'_{\Psi(i)}))$ are isotopic as well.

We may choose a sufficiently small $\epsilon > 0$ so that the construction in \cref{eq: gamma_alpha} can be applied uniformly to each arc arising as a time slice in the isotopy. Then $\Phi(\gamma(\alpha_i))$ and $\gamma(\alpha'_{\Psi(i)})$ are isotopic as loops for each $i = 0 \cdots, l$. Therefore, the equivalence $(\Phi, \Psi)$ induces a mark equivalence between $f_2(A)$ and $f_2(A')$, identifying $\gamma_0$ with $\gamma'_0$.
\end{rmk}

\begin{rmk}
The mutual invertibility of $f_1$ and $f_2$ follows from their construction.

On the one hand, consider the composition $f_2 \circ f_1 : \mathcal{M} \rightarrow \mathcal{M}$ sends each loop $\gamma$ in a marked complete $1$-system of loops $(L, \overrightarrow{\gamma_0}) = (\{\gamma_i\}_{i=0}^l, \overrightarrow{\gamma_0})$ to $\gamma(\alpha(\gamma))$. Since $\alpha^\dagger(\gamma)$ and $\alpha^\dagger(\gamma(\alpha(\gamma)))$ are isotopic (as arcs), we may again choose a sufficiently small $\epsilon > 0$ so that the construction in \cref{eq: gamma_alpha} can be applied uniformly to each arc arising as a time slice in the isotopy. It then follows that $\gamma$ and $\gamma(\alpha(\gamma))$ are isotopic (as loops), and hence $f_2 \circ f_1$ acts as the identity on $\mathcal{M}$.

On the other hand, the composition $f_1 \circ f_2 : \mathcal{W} \rightarrow \mathcal{W}$ sends each arc $\alpha$ in a watermelon $A$ to $\alpha(\gamma(\alpha))$. Since $\alpha$ is isotopic to $\alpha^\dagger(\gamma(\alpha))$ (as an arc), and $\alpha^\dagger(\gamma(\alpha))$ is isotopic to $\alpha(\gamma(\alpha))$, it follows that $\alpha$ and $\alpha(\gamma(\alpha))$ are isotopic. We thus conclude that $f_1 \circ f_2$ is the identity map on $\mathcal{W}$.

This completes the proof of \cref{thm: watermelons and marked loops systems}.
\end{rmk}

\begin{cor}
\label{MaxWDisMaxLoops}
If a watermelon $A$ is maximal, then $L_A$ is maximal. Conversely, if a (marked) complete $1$-system of loops $L$ on $D_n / {\sim}$ is maximal, then $A_L$ is maximal as well.
\end{cor}

\subsection{Standard maximal watermelons}

\begin{pro}
\label{arc = punctures partition}
Let $\mathcal{P}$ be the set of punctures of $D_n$ ($n \geq 2$) and $A$ be a watermelon on $D_n$. Define the map
\begin{equation}
\operatorname{Par} : \left\{[\alpha] \mid \alpha \in A \right\} \rightarrow \left\{ \{\mathcal{P}_1,\mathcal{P}_2\} \mid \mathcal{P}_1 \sqcup \mathcal{P}_2 = \mathcal{P}, \mathcal{P}_i \neq \emptyset ,i=1,2 \right\}, \notag
\end{equation}
which maps from the set of isotopy classes of arcs in $A$ to the set of nontrivial bipartition of $\mathcal{P}$ as follows: any arc $\alpha \in A$ separates $D_n$ into two components $F_{\alpha,1}$ and $F_{\alpha,2}$, let $\mathcal{P}_{\alpha,i}$ be the punctures in $F_{\alpha,i}$, and define $\operatorname{Par}(\alpha) := \{ \mathcal{P}_{\alpha,1},\mathcal{P}_{\alpha,2} \}$. Then $\operatorname{Par}$ is well-defined and injective.
\end{pro}
\begin{proof}
We have that $\operatorname{Par}$ is well-defined up to isotopy of $\alpha$, since the partition will not change when we move one arc to another continuously. Now we prove the $\operatorname{Par}$ is injective as follows. Let $\alpha,\beta$ be two arcs in $A$ such that $\operatorname{Par}([\alpha]) = \operatorname{Par}([\beta])$. Since $A$ is a $1$-system of arcs, the intersection number between $\alpha$ and $\beta$ is $1$ or $0$. If $\alpha$ and $\beta$ are disjoint, they separate $D_n$ into three components, where the only component adjacent to both of the other two contains no punctures. Thus, $\alpha$ and $\beta$ are isotopic. If $\alpha$ and $\beta$ intersect exactly once, they separate $D_n$ into four components, and one pair of opposite components contains no punctures. Hence, they form two half-bigons, which again implies that $\alpha$ and $\beta$ are isotopic.
\end{proof}

\begin{defi}[Short arcs]
\label{Short arcs}
Let $A$ be a watermelon on $D_n$. We denote the collection of all punctures of $D_n$ by $\mathcal{P}$. We say an arc $\alpha$ in $A$ is a \emph{short arc} if one of the sets $\mathcal{P}_{\alpha,1}$ and $\mathcal{P}_{\alpha,2}$ (defined in \cref{arc = punctures partition}) contains exactly one puncture, and the other contains $n-1$ punctures. When $n \geq 3$, only one of the the sets $\mathcal{P}_{\alpha,1}$ and $\mathcal{P}_{\alpha,2}$ consists of a single puncture, and we refer to this puncture as the \emph{puncture isolated by $\alpha$}.
\end{defi}

\begin{lem}
\label{short arcs are disjoint to any others}
Let $A$ be a watermelon on $D_n$ ($n \geq 2$) and let $\alpha \in A$ be a short arc. Then $\alpha$ is disjoint from every other arc in $A$.
\end{lem}

\begin{proof}

If there is another arc $\beta$ in $A$ such that the intersection number between $\alpha$ and $\beta$ is exactly $1$, then $\beta$ separates $F_{\alpha,1}$ (defined in \cref{arc = punctures partition}), which contains only one puncture, into two components. Hence, one of these components contains no punctures and forms a half-bigon (see \cref{HalfBigon}). This contradicts $\alpha$ and $\beta$ being in minimal position.

\begin{figure}[H]
\centering
\includegraphics[width=1.0\textwidth]{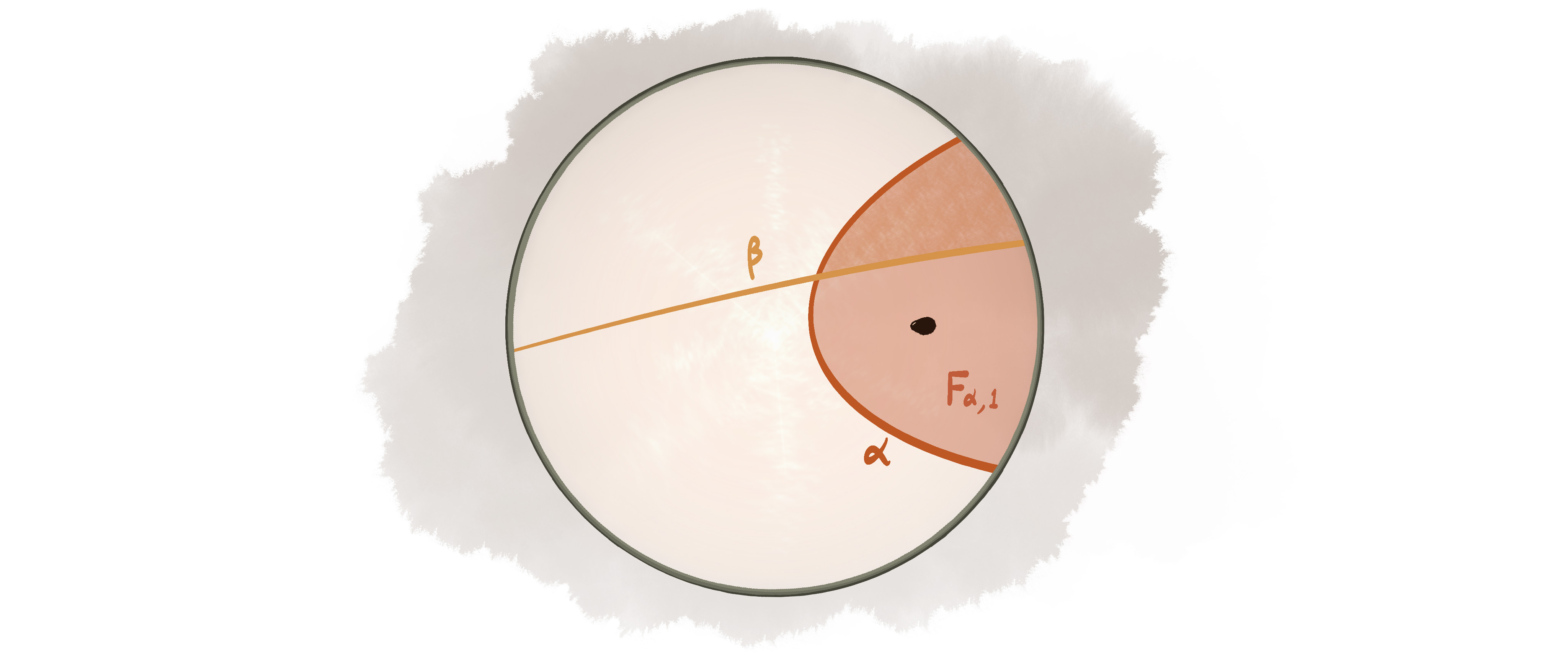}
\caption{If an arc intersects with $\alpha$, they will form a half-bigon (the darkest region in the diagram), which contradicts the minimal position property.}
\label{HalfBigon}
\end{figure}
\end{proof}

\begin{defi}[Standard maximal watermelon]
\label{defi: Standard maximal watermelon}
A maximal watermelon $A$ on $D_n$ is called \emph{standard} if $A$ contains exactly $n$ short arcs. We say $L_A$ is a \emph{standard maximal complete $1$-system of loops}.
\end{defi}

\subsection{Existence and uniqueness of maximal watermelon with n short arcs}

To construct a standard maximal watermelon, we first determine the number of arcs in a standard maximal watermelon.

\begin{pro}
\label{cardinality of maximal watermelon}
For any $n \geq 2$, the cardinality of a maximal watermelon on $D_n$ is $\tfrac{1}{2} n(n-1)$.
\end{pro}
\begin{proof}
This result follows immediately from  \cite[Corollary 1.6]{chen2024systems} and \cref{MaxWDisMaxLoops}.
\end{proof}

\begin{pro}
\label{Existence of standard maximal watermelon}
For any $n \geq 2$, there is a standard maximal watermelon on $D_n$.
\end{pro}

\begin{proof}
Without loss of generality, let
\begin{equation}
\mathcal{P}=\left\{P_i = \tfrac{1}{2}e^{\frac{2(i-1)\pi \sqrt{-1}}{n}} \middle| i = 1 , \cdots, n \right\} \subset C_{1/2} := \{ z \in \mathbb{C} \mid |z|= 1/2 \} \notag
\end{equation}
be the collection of punctures of $D_n$. For every pair $(i,j)$ such that $1 \leq i < j \leq n$, we denote $z_{i,j} := \tfrac{1}{2} \exp{\left( \left( \frac{j-2}{n} + \frac{j-i}{n^2}\right) 2 \pi \sqrt{-1}\right)}$ and $w_{i,j} := \frac{1}{2} \exp{\left( \left( \frac{i-2}{n} + \frac{n-(j-i)}{n^2}\right) 2 \pi \sqrt{-1}\right)}$. We define the arcs $\alpha_{i,j}$ by
\begin{align}
\alpha_{i,j}(t) =
\begin{cases}
(2-3t) z_{i,j} ,& t \in \left[0,\tfrac{1}{3} \right], \\
(2-3t) z_{i,j} + (3t-1) w_{i,j} ,& t \in \left[\tfrac{1}{3},\tfrac{2}{3} \right], \\
(3t-1) w_{i,j} ,& t \in \left[\tfrac{2}{3},1 \right]. \\
\end{cases}
\end{align}
We define the collection of those arcs by $A = \left\{\alpha_{i,j} \mid 1\leq i < j \leq n \right\}$, and later show that $A$ is a standard maximal watermelon. The following picture shows an example when $n=4$, where a standard maximal watermelon has six arcs (\cref{StandardMaxDiagramFourPunctures}).

\begin{figure}[H]
\centering
\includegraphics[width=1.0\textwidth]{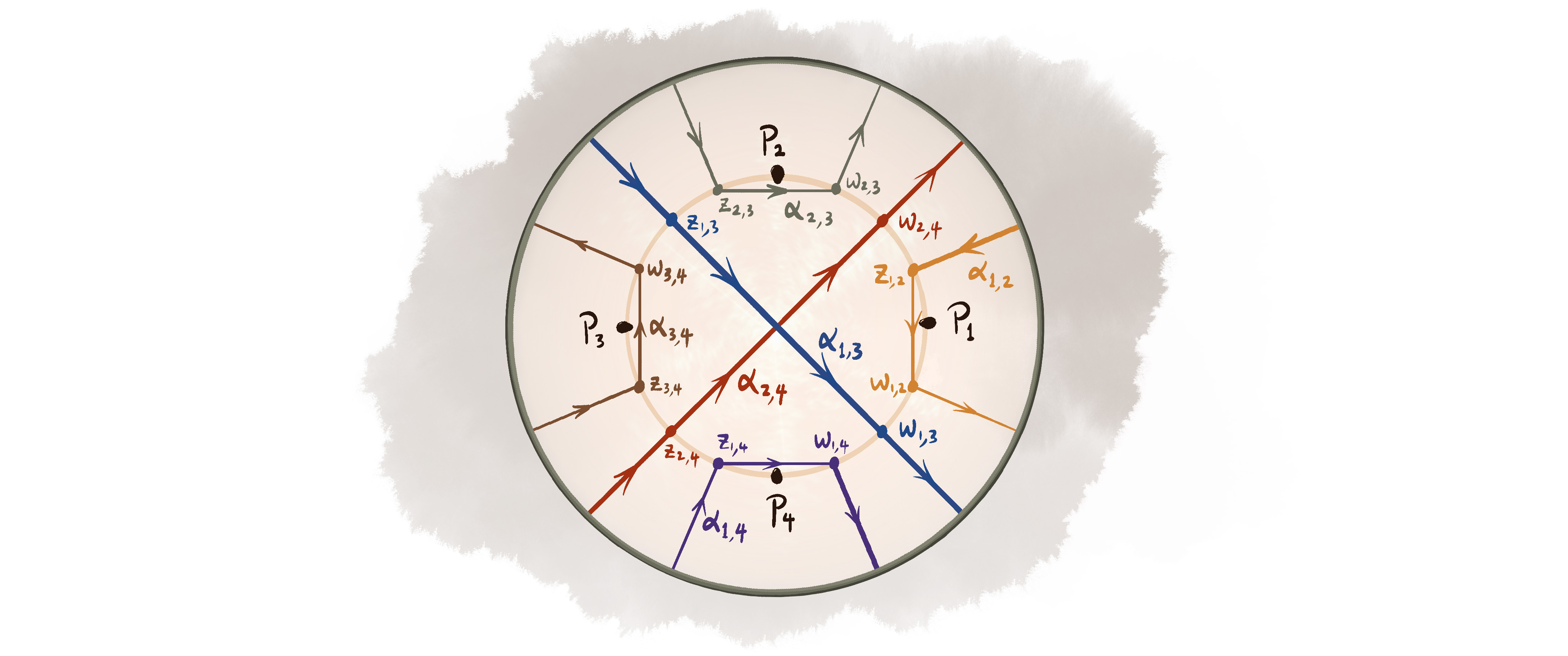}
\caption{A standard maximal watermelon when $n=4$.}
\label{StandardMaxDiagramFourPunctures}
\end{figure}

We first show that $A$ is a watermelon. Since each arc $\alpha_{i,j}$ separates the disk $D_n$ into two components, each containing at least one puncture, it follows that $\alpha_{i,j}$ is essential. All intersection points between arcs in $A$ lie within the open disk $\mathbb{B}_{1/2} := \left\{z \in \mathbb{C} \mid |z|< 1/2 \right\}$. Since the arcs restrict to straight lines in $\mathbb{B}_{1/2}$, the intersection number between any pair of arcs is at most one. Furthermore, since the arcs in $A$ induce distinct partitions of the punctures, they are non-isotopic by the proof of \cref{arc = punctures partition}. Moreover, by construction, for any pair of intersecting arcs, the four segments into which they divide $C_{1/2}$ each contain punctures. Therefore, no pair of arcs in $A$ forms a half-bigon, and hence they are in minimal position. Therefore, $A$ is a watermelon.

The number of arcs in $A$ is the number of ordered pairs $(i,j)$ satisfying $1\leq i < j \leq n$, which is $\tfrac{1}{2} n(n-1)$. There are $n$ short arcs $\left\{\alpha_{i,i+1} \mid 1 \leq i \leq n-1 \right\} \cup \{\alpha_{1,n}\}$. Thus, $A$ is maximal and standard.

\end{proof}

\begin{pro}
\label{watermelon with n short arcs is sub-watermelon of the standard max one}
For any $n \geq 3$, if a watermelon on $D_n$ contains $n$ short arcs, then it is equivalent to a \emph{sub-watermelon}\footnote{A \emph{sub-watermelon} refers to a sub-collection of arcs of some watermelon. Sub-watermelons are evidently watermelons.} of the standard maximal watermelon constructed in \cref{Existence of standard maximal watermelon}.
\end{pro}

\begin{proof}

Let $A$ be a watermelon on $D_n$, and let $\alpha_1 , \cdots, \alpha_n$ be the short arcs in $A$. Denote the component of the complement of $\alpha_i$ that contains only one puncture by $F_{\alpha_i , 1}$. By \cref{short arcs are disjoint to any others}, $F := D_n \setminus \left( \cup_{i=1}^n F_{\alpha_i,1} \right)$ is a topological disk where any non-short arc in $A$ lies on $F$ and $F \cap \partial D_n$ consists of $n$ segments, which we denote by $S_1 , \cdots , S_n$.

For any non-short arc $\alpha$ in $A$, its two endpoints lie on distinct boundary components $S_i$ and $S_j$. Moreover, two arcs in $F$ with endpoints lying on $S_1, \ldots, S_n$ are isotopic by an isotopy which setwise preserves $S_1, \ldots, S_n$ if and only if their endpoints lie on the same pair of segments. Hence, every arc in $A$ is isotopic to an arc in the standard maximal watermelon we constructed in \cref{Existence of standard maximal watermelon}.

\end{proof}

\begin{cor}
\label{Uniqueness of maximal watermelon with n short arcs}
For any $n \geq 3$, if a maximal watermelon on $D_n$ contains $n$ short arcs, then it is equivalent to the standard maximal watermelon constructed in \cref{Existence of standard maximal watermelon}.
\end{cor}

\begin{nota}
Given the existence and uniqueness of the standard maximal watermelon (as established above), we denote it by $A^s$, and denote the standard maximal complete $1$-system of loops by $L^s$.
\end{nota}

\section{Uniqueness of Maximal Complete $1$-System on $N_{1,n}$ when $n \leq 5$}\label{section6}

In this section, we prove the ``if'' direction of main theorem \cref{main}: maximal complete $1$-systems of loops on $N_{1,n}$ are unique up to the action of the mapping class group \emph{if} $n \leq 5$. We treat the cases $n = 2, 3, 4, 5$ individually (see \cref{thm: Uniqueness when n = 2}, \cref{Thm: Uniqueness when n = 3}, \cref{Thm: Uniqueness when n = 4}, and \cref{Thm: Uniqueness when n = 5}). The case $n \geq 6$, which corresponds to the ``only if'' direction, is handled separately in \cref{thm: NonUniquenessNGeqSix}.

In all relevant theorems below, we will begin by enumerating the maximal watermelons up to watermelon equivalence. Each such watermelon corresponds to a maximal marked complete $1$-system of loops (see \cref{thm: watermelons and marked loops systems} and \cref{MaxWDisMaxLoops}). By forgetting the marks, we then study whether the resulting maximal complete $1$-systems of loops are all equivalent.

\begin{rmk}
In fact, the proof for the case $n = 1$ is already implicitly contained in \cite[Section 5.1]{chen2024systems}, where the approach is to directly consider the system of loops. In the following, we consider the cases $n = 2, \cdots, 5$, using watermelons.
\end{rmk}

Recall that the cardinality of a maximal watermelon on $D_n$ is $\tfrac{1}{2} n(n - 1)$ (\cref{cardinality of maximal watermelon}). Therefore, when $n = 2, 3, 4, 5$, the corresponding cardinalities of the maximal watermelons are $1, 3, 6, 10$, respectively. We will assume these numbers in the following discussion.

In the following proofs, we will frequently use \cref{arc = punctures partition}. That is, different arcs in a maximal watermelon correspond to different partitions of the punctures in $D_n$. Moreover, the total number of possible ways to partition the punctures gives an upper bound on the cardinality of a maximal watermelon.

\subsection{Uniqueness when n is two or three}

\begin{thm}
\label{thm: Uniqueness when n = 2}
Maximal complete $1$-systems of loops on $N_{1,2}$ are unique up to equivalence.
\end{thm}

\begin{proof}
A maximal watermelon on $D_2$ consists of a single arc, and it is unique, since it can only consist of the arc that separates the two punctures. Consequently, the maximal marked complete $1$-system of loops on $D_2 / {\sim}$ is also unique, and thus the maximal complete $1$-system of loops on $D_2 / {\sim}$ is unique as well.

% We regard $N_{1,2}$ as $D_2 / {\sim}$ as \cref{c4section: Watermelons and Standard Maximal One System}. When $n = 2$, the cardinality of the maximal complete $1$-system of loops is $2$. Let $L = \{ \gamma_0 , \gamma_1 \}$ be a maximal complete $1$-system of loops. Up to the action of the mapping class group action, we may position $\gamma_0$ on $\partial D_2 / {\sim}$. Then $\alpha_{\gamma_1}$ (the definition is in \cref{Associated watermelons}) is a simple arc, which is the unique arc that separates two punctures (see \cref{UniquenessWhenNIsTwo}).

\begin{figure}[H]
\centering
\includegraphics[width=1.0\textwidth]{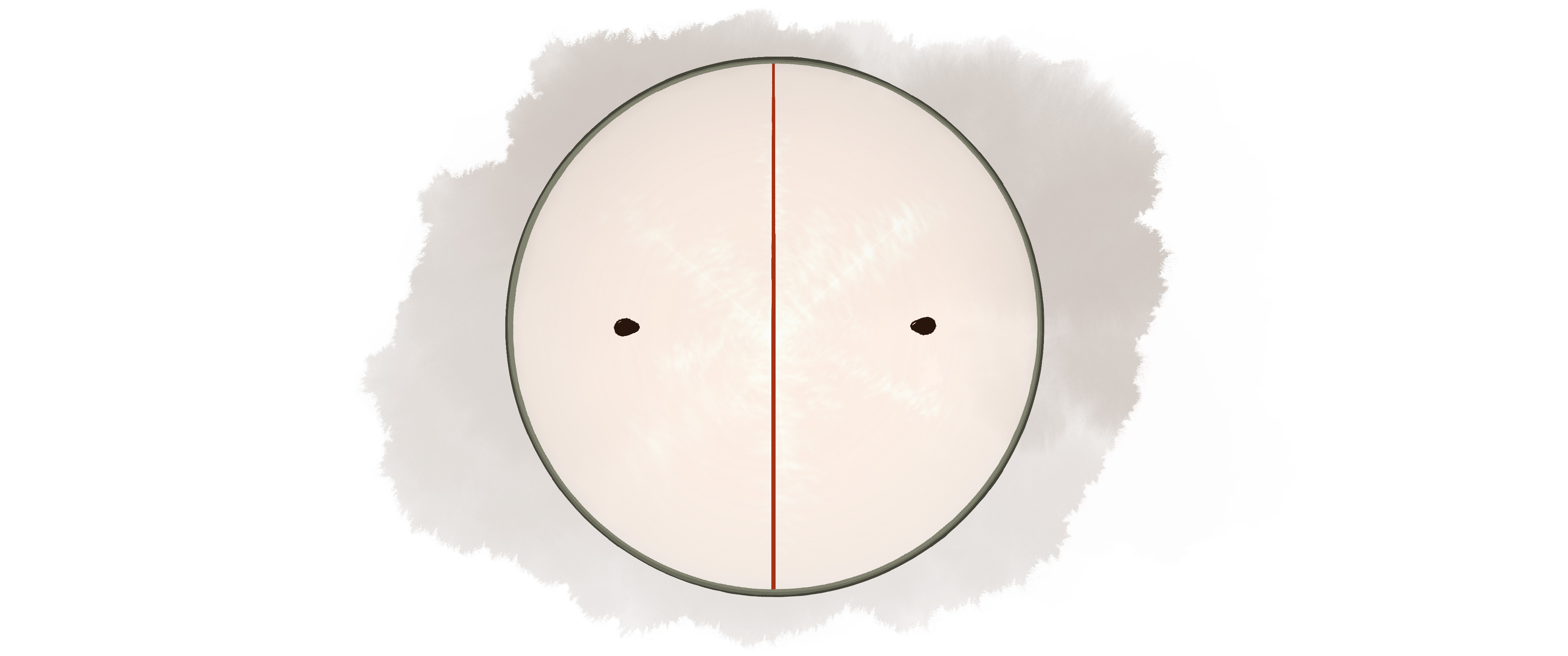}
\caption{The unique maximal watermelon on $D_2$.}
\label{UniquenessWhenNIsTwo}
\end{figure}

\end{proof}

\begin{thm}
\label{Thm: Uniqueness when n = 3}
Maximal complete $1$-systems of loops on $N_{1,3}$ are unique up to equivalence.
\end{thm}

\begin{proof}
A maximal watermelon $A$ on $D_3$ consists of three arcs, which we denote by $A = { \alpha_1, \alpha_2, \alpha_3 }$. Since $\alpha_1$, $\alpha_2$, and $\alpha_3$ are non-isotopic short arcs, each of them separates $D_3$ into two components, one of which contains exactly one puncture. By \cref{Uniqueness of maximal watermelon with n short arcs}, $A$ is unique up to watermelon equivalence (see also \cref{UniquenessWhenNIsThree}). Hence, the maximal complete $1$-system of loops on $D_3 / {\sim}$ is unique as well.

\begin{figure}[H]
\centering
\includegraphics[width=1.0\textwidth]{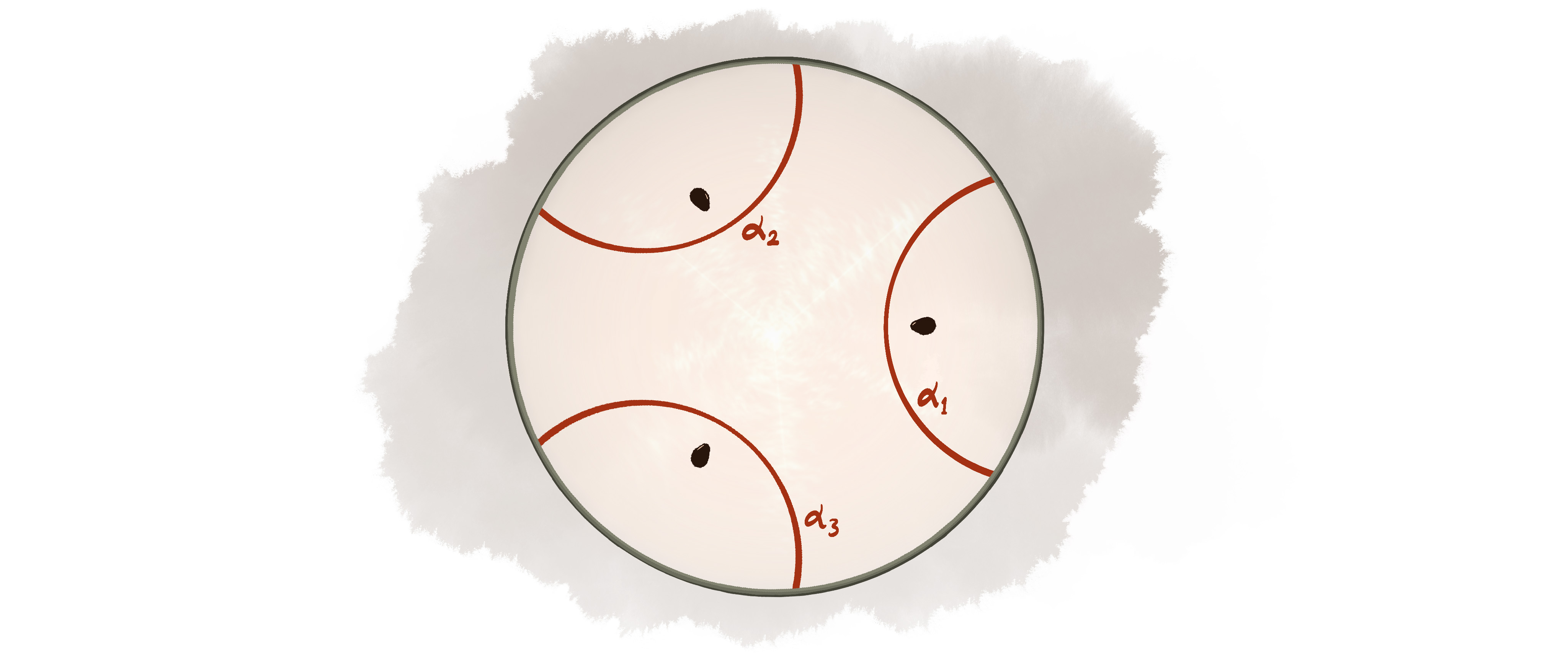}
\caption{The unique maximal watermelon on $D_3$.}
\label{UniquenessWhenNIsThree}
\end{figure}

\end{proof}

\subsection{Uniqueness when n is four}

\begin{thm}
\label{Thm: Uniqueness when n = 4}
Maximal complete $1$-systems of loops on $N_{1,4}$ are unique up to equivalence.
\end{thm}

\begin{proof}
A maximal watermelon $A$ on $D_4$ consists of six arcs, which we denote by $A = \{ \alpha_1, \cdots, \alpha_6 \}$. Let $\mathcal{P} = \{P_1, P_2, P_3, P_4 \}$ be the four punctures of $D_4$. Since there are only three ways to bipartition $\mathcal{P}$ into two subsets, each containing exactly two elements, there can be at most three non-short arcs in $A$. It follows that $A$ contains at least three short arcs (and at most four).

By \cref{Uniqueness of maximal watermelon with n short arcs}, a maximal watermelon with four short arcs is unique up to watermelon equivalence. It therefore suffices to show that a maximal complete $1$-system of loops $L_A$ associated with a watermelon $A$ with exactly three short arcs is equivalent to the standard maximal complete $1$-system of loops $L^s$.

\begin{figure}[H]
\centering
\includegraphics[width=1.0\textwidth]{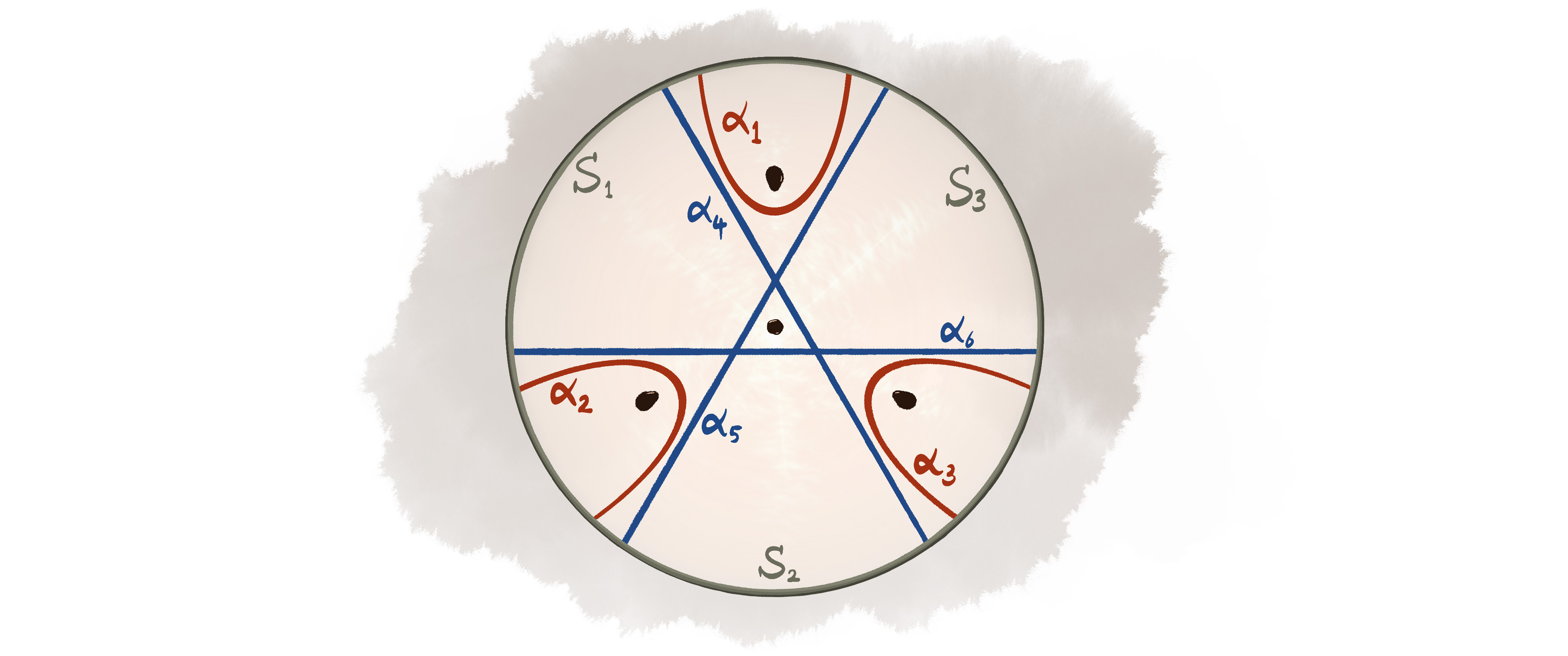}
\caption{The maximal watermelon with three short arcs on $D_4$.}
\label{THreeShortArcsSystemWhenNIsFour}
\end{figure}

\begin{figure}[H]
\centering
\includegraphics[width=1.0\textwidth]{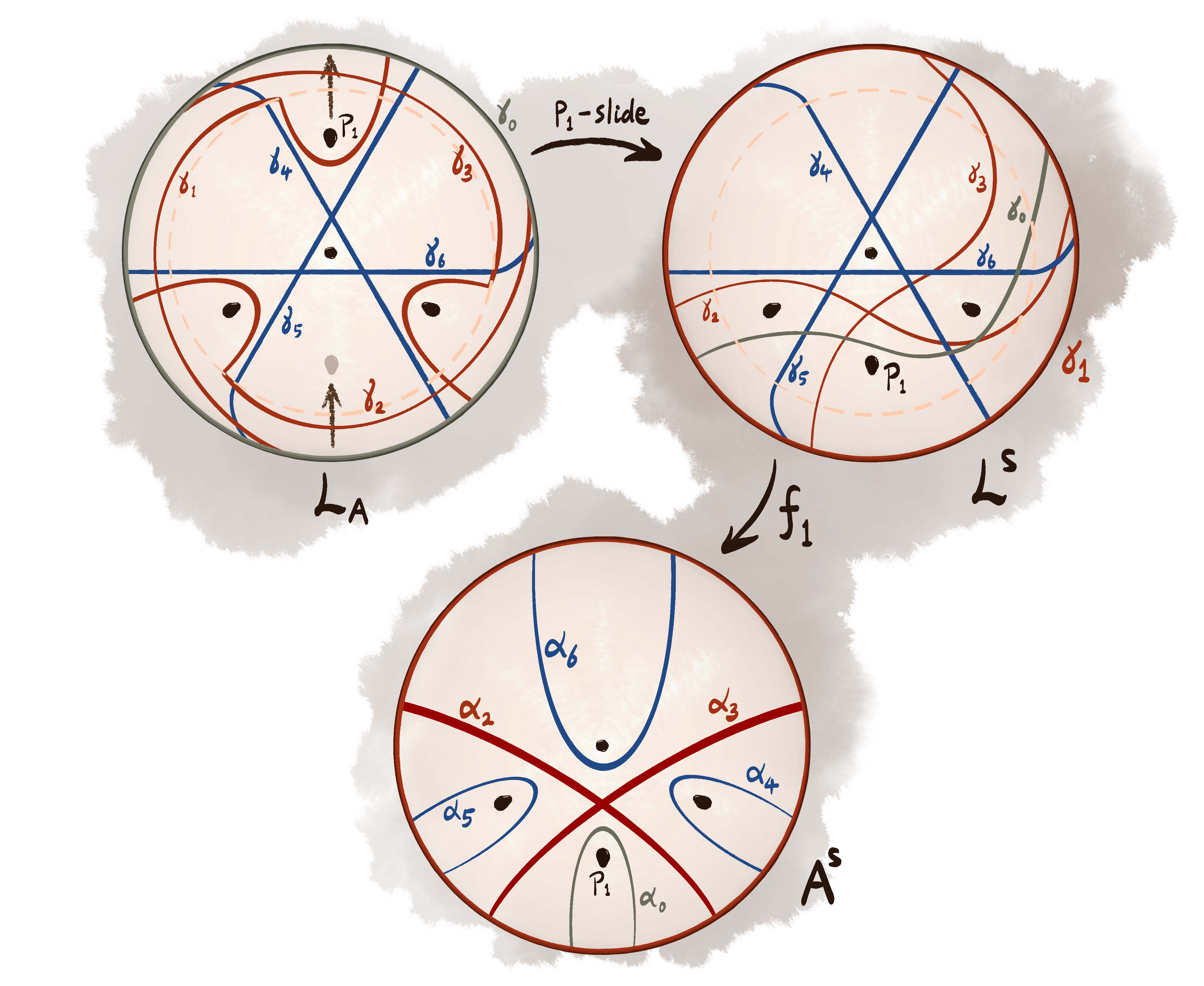}
\caption{Puncture sliding $L_A$ to $L^s$ when $n=4$.}
\label{UniquenessWhenNIsFour}
\end{figure}

Let $\alpha_1, \alpha_2, \alpha_3$ denote the short arcs in $A$, and $\alpha_4, \alpha_5, \alpha_6$ the non-short arcs. By \cref{short arcs are disjoint to any others}, the arcs $\alpha_4, \alpha_5, \alpha_6$ lie in the region $F := D_n \setminus \left( \cup_{i=1}^3 F_{\alpha_i,1} \right)$, which is topologically a punctured disk. Denote by $S_1, S_2, S_3$ the three segments of $F \cap \partial D_n$. Since $F$ contains one puncture, simple arcs in $F$ connecting $S_i$ and $S_j$ (where $i$ and $j$ are not necessarily distinct) fall into exactly two isotopy classes. In total, there are nine nontrivial simple arcs in $F$, up to isotopies which setwise preserve $S_1, S_2$ and $S_3$. Among them, exactly three classes contain non-short arcs as representatives. The arcs $\alpha_4, \alpha_5, \alpha_6$ respectively lie in these three isotopy classes. Hence, all arcs in $A$ are determined (see \cref{THreeShortArcsSystemWhenNIsFour} for an illustration of $A$).

For $L_A$, we perform a puncture slide of $P_1$ as described in \cref{UniquenessWhenNIsFour} (see also \cref{Defi: PunctureSlides}; here, the loops $\gamma_i := \gamma(\alpha_i)$ in $L_A$ are associated loops of $\alpha_i$ in the watermelon $A$). After the puncture slide, we obtain that $L_A$ is equivalent to the standard maximal complete $1$-system of loops $L^s$.

\end{proof}

\subsection{Uniqueness when n is five}

\begin{defi}[P-parallel]
\label{defi: P-parallel}
Let $\alpha_1, \alpha_2$ be two arcs on $D_n$ with $n \geq 3$, and let $P$ be a puncture of $D_n$. We denote by $D_{n-1} := D_n \cup \{P\}$ the surface obtained from $D_n$ by filling in the puncture $P$. We say that $\alpha_1$ and $\alpha_2$ are \emph{$P$-parallel} if they are isotopic on $D_{n-1}$.
\end{defi}

\begin{lem}
\label{lem: fill point at most two arcs are isotopic}
Let $A$ be a watermelon on $D_n$ with $n \geq 2$. If two distinct arcs in $A$ are $P$-parallel, then they are disjoint. Moreover, at most two distinct arcs in $A$ can be $P$-parallel.
\end{lem}      

\begin{proof}
Let $\alpha_1$ and $\alpha_2$ be two distinct $P$-parallel arcs in $A$. Suppose for contradiction that $\alpha_1$ and $\alpha_2$ intersect. Then they divide $D_n$ into four regions. Since $\alpha_1$ and $\alpha_2$ are in minimal position on $D_n$, each of these four regions must contain at least one puncture. On the other hand, because $\alpha_1$ and $\alpha_2$ are $P$-parallel, the region that originally contains $P$ contains no other punctures. After filling in $P$, this region becomes a half-bigon, which allows $\alpha_1$ and $\alpha_2$ to be isotoped to be disjoint. Moreover, for them to be isotopic to each other, the region on the opposite side must also be a half-bigon. This contradicts the earlier observation that all four regions initially contained punctures. Hence, $\alpha_1$ and $\alpha_2$ cannot intersect.

Now suppose that there are three arcs $\alpha_1, \alpha_2, \alpha_3$ in $A$ that are pairwise $P$-parallel. Then they must be pairwise disjoint. In particular, $\alpha_1$ and $\alpha_2$ divide $D_n$ into three regions: let $F_1$ be the region bounded by $\alpha_1$ and $\partial D_n$, $F_2$ the region bounded by $\alpha_2$ and $\partial D_n$, and $F_3$ the region between $\alpha_1$ and $\alpha_2$. As above,  each of these three regions must contain punctures, and $F_3$ (the middle region) contains $P$ and only $P$. The same conclusion also holds for the pairs $(\alpha_1, \alpha_3)$ and $(\alpha_2, \alpha_3)$.

Consider now the possible positions of $\alpha_3$. We will show that, regardless of where $\alpha_3$ lies, there always exist indices $1 \leq i < j \leq 3$ such that the region between $\alpha_i$ and $\alpha_j$ does not contain $P$, and hence a contradiction is obtained. If $\alpha_3$ lies in $F_1$, then the region between $\alpha_1$ and $\alpha_3$ does not contain $P$, contradicting the conclusion. Similarly, $\alpha_3$ cannot lie in $F_2$. If $\alpha_3$ lies in $F_3$, then at least one of the following holds: the region between $\alpha_1$ and $\alpha_3$ does not contain $P$; the region between $\alpha_2$ and $\alpha_3$ does not contain $P$; or $\alpha_3$ is null-homotopic, which contradicts the definition of a watermelon, as no arc in a watermelon is null-homotopic.
\end{proof}

\begin{defi}[$P$-reduced watermelons]
\label{defi: $P$-reduced watermelons}
Let $A$ be a watermelon on $D_n$ with $n \geq 3$, and let $P$ be a puncture of $D_n$. Set $D_{n-1} := D_n \cup \{P\}$. Consider the partition of $A$ by $P$-parallelism, excluding the classes consisting of arcs that are null-isotopic on $D_{n-1}$:
\begin{align}
\mathcal{A} :=& \left\{ \left\{ \alpha' \in A \mid \alpha' \text{ is } P\text{-parallel to } \alpha \right\} \mid \alpha \in A, \ \alpha \text{ is not null-isotopic on } D_{n-1} \right\} \notag \\          
:=& \{A_1, \cdots, A_m\}. \notag
\end{align}
For each $i = 1, \cdots, m$, choose a representative arc $\alpha'_i$ such that $\alpha'_i$ is isotopic on $D_{n-1}$ to the arcs in $A_i$, and such that the set $A' := \{\alpha'_i \mid i = 1, \cdots, m\}$ is in pairwise minimal position on $D_{n-1}$. Then $A'$ forms a watermelon on $D_{n-1}$, and we call $A'$ a \emph{$P$-reduced watermelon} of $A$. We refer to $A_i$ as the \emph{$P$-source set} of $\alpha'_i$, and $\alpha'_i$ as a \emph{$P$-reduced arc} of $A_i$.
\end{defi}

\begin{rmk}
\label{rmk: each $P$-source set contains either one or two elements}
By \cref{lem: fill point at most two arcs are isotopic}, each $P$-source set contains either one or two elements.
\end{rmk}

\begin{defi}[Saturated watermelons]
\label{defi: Saturated watermelons}
Let $A$ be a watermelon on $D_n$. We say that $A$ is \emph{saturated}\footnote{This term is adapted from \cite{MR4900491}.} on $D_n$ if there is no simple arc $\alpha$ on $D_n$ such that $\{\alpha\} \cup A$ is a watermelon.
\end{defi}

\begin{lem}
\label{lem2: maximal system after filling is saturated}
Let $A$ be a maximal watermelon on $D_n$ with $n \geq 3$, and let $P$ be a puncture of $D_n$. Then any $P$-reduced watermelon $A'$ of $A$ is saturated on $D_{n-1}$.
\end{lem}

\begin{proof}
Assume that there is a simple arc $\alpha'$ such that $A' \cup \{ \alpha' \}$ is a watermelon on $D_{n-1}$. If $\alpha'$ passes through $P$, apply a slight perturbation to $\alpha'$ so that it no longer passes through $P$ and $A' \cup \{ \alpha' \}$ is still a watermelon on $D_{n-1}$. There is another arc $\alpha''$ that is isotopic to $\alpha'$ in $D_{n-1}$ such that $\alpha'$ and $\alpha''$ are disjoint, and $P$ lies between $\alpha'$ and $\alpha''$. The construction process of $\alpha''$ can be as follows: take a path connecting $P$ to a point on $\alpha'$ that does not pass through any other punctures. Define $\alpha''$ as follows: first follow $\alpha'$, then along the path to $P$, loop around $P$, and then return along the path before completing the remaining part of $\alpha'$ (see the first figure in \cref{MaxAfterFillIsTaut}). We may assume that every pair of arcs in $A \cup \{\alpha', \alpha''\}$ is in minimal position, since we may equip $D_n$ with a hyperbolic structure such that $\partial D_n$ is geodesic and choose geodesic representatives for the arcs in $A \cup \{\alpha', \alpha''\}$.

\begin{figure}[H]
\centering
\includegraphics[width=1.0\textwidth]{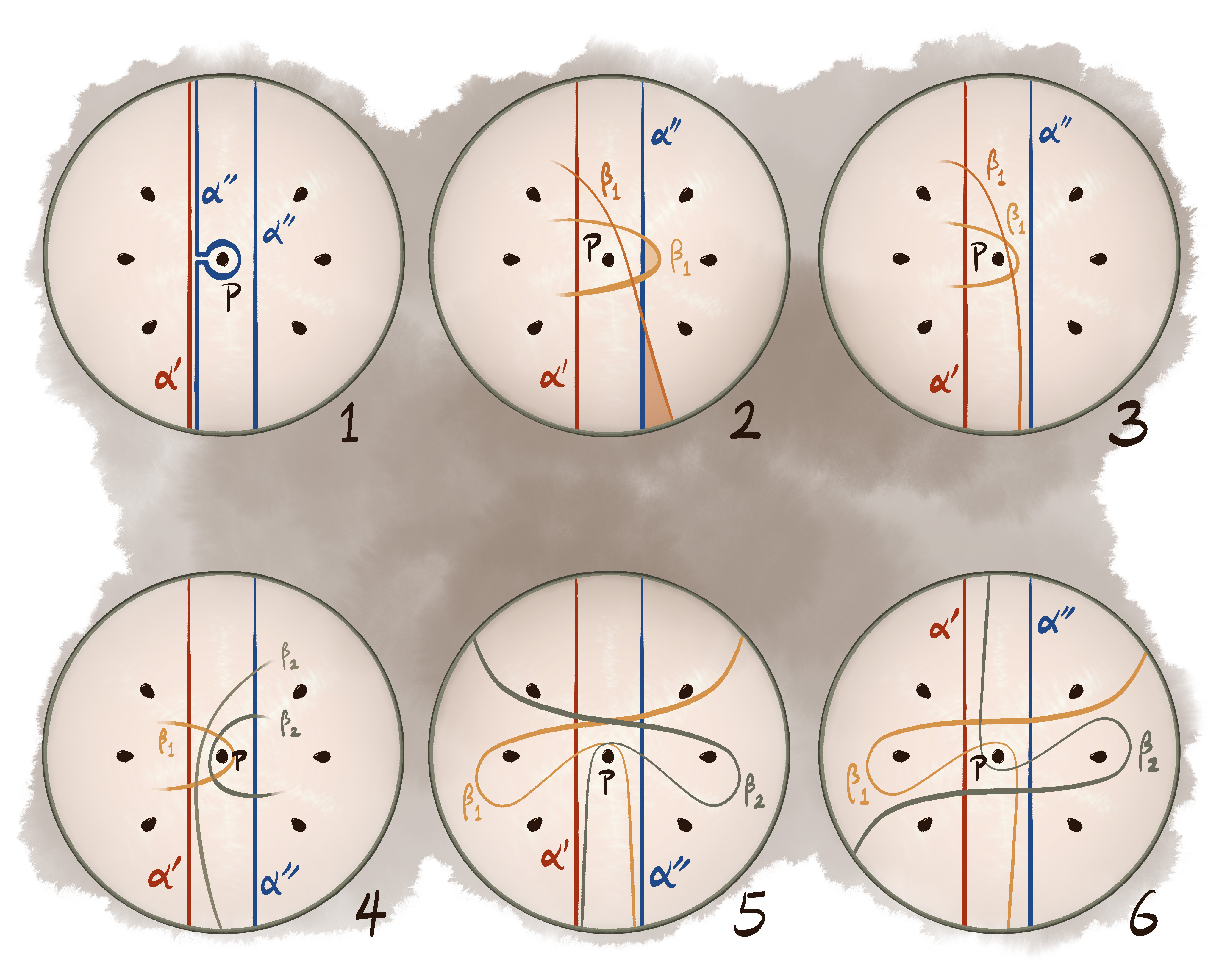}
\caption{The diagrams in the proof of \cref{lem2: maximal system after filling is saturated}.}
\label{MaxAfterFillIsTaut}
\end{figure}

We now show that at least one of $\alpha'$ and $\alpha''$ intersects every arc in $A$ at most once, which contradicts the maximality of $A$. Suppose there exist $\beta_1, \beta_2 \in A$ such that $\beta_1$ intersects $\alpha'$ at least twice, and $\beta_2$ intersects $\alpha''$ at least twice. Since the geometric intersection number of the isotopy classes $[\beta_1]$ and $[\alpha']$ on $D_{n-1}$ is at most $1$, the arcs $\beta_1$ and $\alpha'$ must form at least one (half-)bigon on $D_{n-1}$, and $P$ lies in this (half-)bigon. Similarly, $\beta_2$ and $\alpha''$ form at least one (half-)bigon on $D_{n-1}$, with $P$ lying inside it.

If $\beta_1$ and $\alpha'$ form at least one (half-)bigon on $D_{n-1}$, we let $S_{\alpha'} \cup S_{\beta_1}$ be the boundary of the (half-)bigon, where $S_{\alpha'}$ is a segment of $\alpha'$ and $S_{\beta_1}$ is a segment of $\beta_1$. Then $S_{\beta_1}$ is in the region between $\alpha'$ and $\alpha''$, otherwise, $\beta_1$ and $\alpha''$ form a (half-)bigon on $D_{n}$, which contradicts to minimal position (see the second and the third figures in \cref{MaxAfterFillIsTaut}). Similarly, $S_{\beta_2}$ is in the region between $\alpha'$ and $\alpha''$.

We first consider the case when $\beta_1$ and $\alpha'$ form a bigon on $D_{n-1}$. Then, regardless of whether $\beta_2$ forms a bigon or a half-bigon with $\alpha''$, $\beta_1$ and $\beta_2$ must intersect at least twice. Similarly, $\beta_2$ and $\alpha''$ cannot form a bigon on $D_{n-1}$ (see the fourth figure in \cref{MaxAfterFillIsTaut}).

If both pairs $\beta_1, \alpha'$ and $\beta_2, \alpha''$ form only half-bigons, then $\beta_1$ and $\alpha'$ intersect exactly twice (since if the intersection number were greater than two, they would necessarily form a bigon on $D_{n-1}$). Similarly, $\beta_2$ and $\alpha''$ also intersect exactly twice. In this case, there are only two possibilities, and both lead to $\beta_1$ and $\beta_2$ intersecting at least twice (see the fifth and sixth figures in \cref{MaxAfterFillIsTaut}).

\end{proof}

\begin{lem}
\label{5 arcs P-parallel}
Let $A$ be a watermelon on $D_5$, and let $P$ be a puncture of $D_5$. Suppose that a $P$-reduced watermelon $A'$ of $A$ has cardinality $5$. Then $A$ is not maximal on $D_5$.
\end{lem}

\begin{proof}
By \cref{arc = punctures partition}, the number of non-short arcs in $A'$ is at most $3$. Hence, we only need to consider the cases where $A'$ contains $4$, $3$, or $2$ short arcs on $D_4$.

\begin{itemize}
\item Suppose there are exactly $4$ short arcs in $A'$. By \cref{watermelon with n short arcs is sub-watermelon of the standard max one}, $A'$ is a sub-watermelon of the standard maximal watermelon on $D_4$. Since $A'$ contains only five arcs while the standard maximal watermelon on $D_4$ contains six, it follows that $A'$ is not saturated on $D_4$. Therefore, by \cref{lem2: maximal system after filling is saturated}, $A$ is not maximal on $D_5$.
\item Suppose there are exactly $3$ short arcs in $A'$ (i.e., $A'$ contains exactly $2$ non-short arcs). If there exists a path connecting the puncture $Q$ and $\partial D_4$ that is disjoint from all arcs in $A'$, then the boundary of a sufficiently small regular neighborhood of this path, denoted by $\beta$, is a short arc that isolates $Q$ and is disjoint from all arcs in $A'$. Therefore, $A' \cup \{\beta\}$ is a watermelon on $D_4$, which implies that $A'$ is not saturated. By \cref{lem2: maximal system after filling is saturated}, it follows that $A$ is not maximal. If there does not exist a path connecting the puncture $Q$ to $\partial D_4$, this means that $Q$ must be fenced off by the two non-short arcs in $A'$ (see the left picture in \cref{ExactlyThreeOrTwoShortArcsInAPrime}). Consequently, the intersection number between these two non-short arcs must be greater than $1$, which contradicts the fact that $A'$ is a watermelon.

\begin{figure}[H]
\centering
\includegraphics[width=1.0\textwidth]{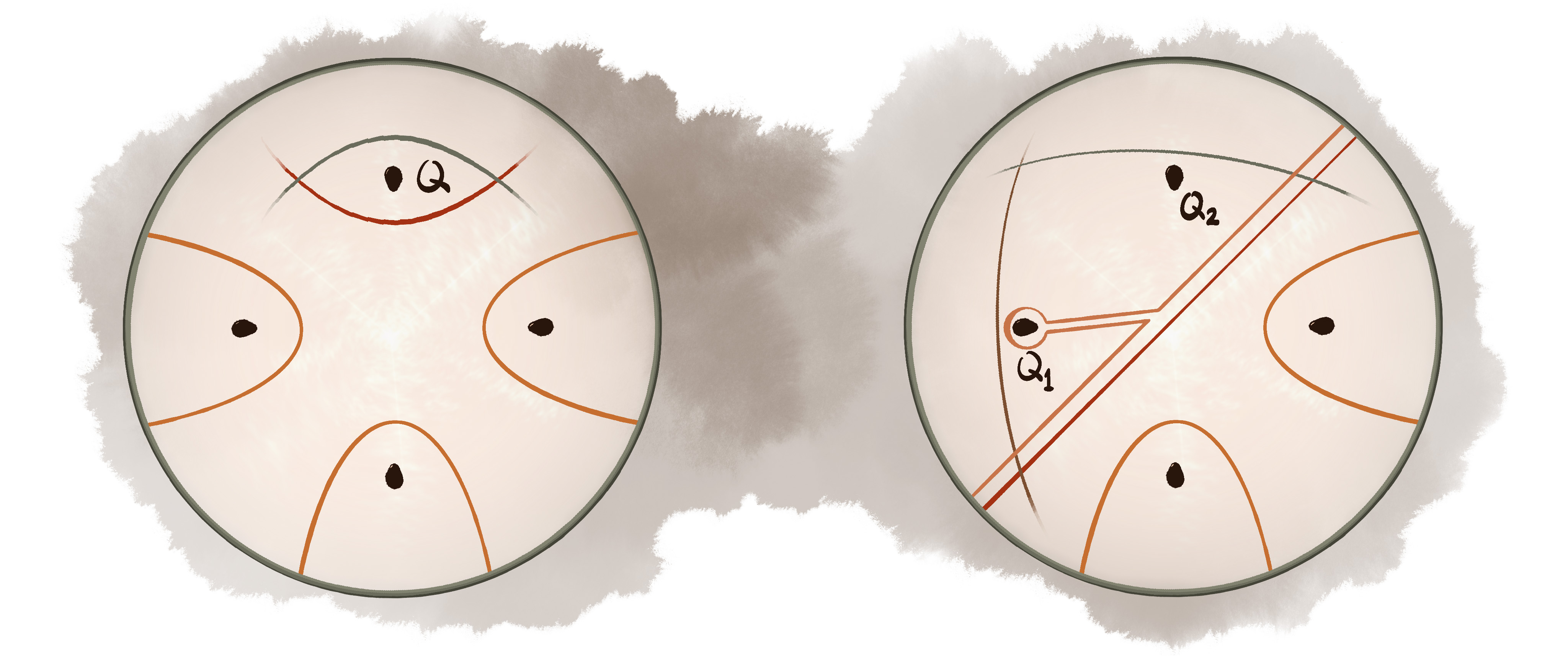}
\caption{The left picture shows that $Q$ is fenced off by two non-short arcs, while the right picture shows that $Q_1$ and $Q_2$ are fenced off by the three non-short arcs.}
\label{ExactlyThreeOrTwoShortArcsInAPrime}
\end{figure}

\item Suppose there are exactly $2$ short arcs in $A'$ (i.e., $A'$ contains exactly $3$ non-short arcs). Then both of the punctures $Q_1$ and $Q_2$ of $D_4$ that are not isolated by short arcs must be fenced off by the three non-short arcs; otherwise, as before, we could add a short arc isolating $Q_1$ or $Q_2$, which would imply that $A'$ is not saturated, and hence $A$ is not maximal. In fact, if these three non-short arcs enclose one or more regions whose boundaries do not contain any segment of $\partial D_4$, the only possibility is a single triangle. In this case, each pair of the three arcs intersects exactly once, and both $Q_1$ and $Q_2$ lie inside this triangle. We can then construct an additional arc that runs close and parallel to one of the non-short arcs outside the triangle and separates $Q_1$ and $Q_2$ within the triangle (see the right picture in \cref{ExactlyThreeOrTwoShortArcsInAPrime}). This again implies that $A'$ is not saturated, and thus $A$ is not maximal.
\end{itemize}
\end{proof}

\begin{lem}
\label{lem: $P$-source set has short arc}
Let $A$ be a maximal watermelon on $D_5$, and let $P$ be a puncture of $D_5$ that is not isolated by any short arc in $A$. Let $A'$ be a $P$-reduced watermelon of $A$. Then for any short arc $\alpha'$ in $A'$ on $D_4$, its $P$-source set (see \cref{defi: $P$-reduced watermelons}) must contain one arc that is a short arc on $D_5$.
\end{lem}

\begin{proof}
We proceed by contradiction. Suppose that none of the arcs in the $P$-source set of $\alpha'$ are short on $D_5$. Let $\alpha$ be an arbitrary arc in the $P$-source set. Denote by $Q$ the puncture that $\alpha'$ isolates on $D_4$. Since $\alpha$ and $\alpha'$ are isotopic on $D_4$, and $\alpha'$ is short on $D_4$, it follows that $\alpha$ is also short on $D_4$. Therefore, $\alpha$ isolates $Q$ on $D_4$. The arc $\alpha$ divides $D_4$ into two components. Let $F_{\alpha,1}$ denote the component that contains exactly the puncture $Q$. Moreover, since $\alpha$ is not short on $D_5$, the puncture $P$ must lie in $F_{\alpha,1}$ (see \cref{AtLeastFourShortArcsWhenNIsFive}).

First, we claim that there is no arc in $A$ entirely contained in $F_{\alpha,1}$. Otherwise, such an arc would either be isotopic to $\alpha$ (which is impossible since $A$ is a watermelon and all its arcs are pairwise non-isotopic), or it would isolate either $P$ or $Q$. The latter is also ruled out: by the assumptions of the lemma, $A$ contains no short arc that isolates $P$; and by the hypothesis of contradiction, $A$ contains no short arc that isolates $Q$ either.

Secondly, we claim that there must exist some arcs in $A$ that intersect $F_{\alpha,1}$ (but are not entirely contained in it), and that, together with $\alpha$, fence off $P$ and $Q$. Indeed, otherwise we could find a short arc $\alpha^{\diamond}$ entirely contained in $F_{\alpha,1}$ that isolates either $P$ or $Q$, so that $A \cup {\alpha^{\diamond}}$ would form a larger watermelon, contradicting the maximality of $A$.

Let $\beta_1 \in A$ be an arc that intersects $F_{\alpha,1}$ but is not entirely contained in it. Then $\beta_1$ intersects $\alpha$ exactly once, and the endpoint of $\beta_1$ that lies in $F_{\alpha,1}$ is located on the boundary segment $F_{\alpha,1} \cap \partial D_5$.

Furthermore, $\beta_1$ must separate $P$ and $Q$; otherwise, $\beta_1$ and $\alpha$ would form a half-bigon within $F_{\alpha,1}$, contradicting the minimal position of arcs in $A$, since $A$ is a watermelon. Therefore, up to watermelon equivalence, we may always draw $\beta_1$ in the configuration shown in \cref{AtLeastFourShortArcsWhenNIsFive}.

However, the arc $\beta_1$ alone does not suffice to fence off the punctures $P$ and $Q$. Nevertheless, it serves as a canonical reference in the diagram, allowing the subsequent analysis to proceed relative to it. The arc $\beta_1$ divides $F_{\alpha,1}$ into two regions, which we denote by $F_1$ and $F_2$, with $Q \in F_1$ and $P \in F_2$. Let $S_1, \dots, S_5$ denote the boundary segments that enclose $F_1$ and $F_2$, as shown in the left picture in \cref{AtLeastFourShortArcsWhenNIsFive}. Topologically, both $F_1$ and $F_2$ are punctured disks.

\begin{figure}[H]
\centering
\includegraphics[width=1.0\textwidth]{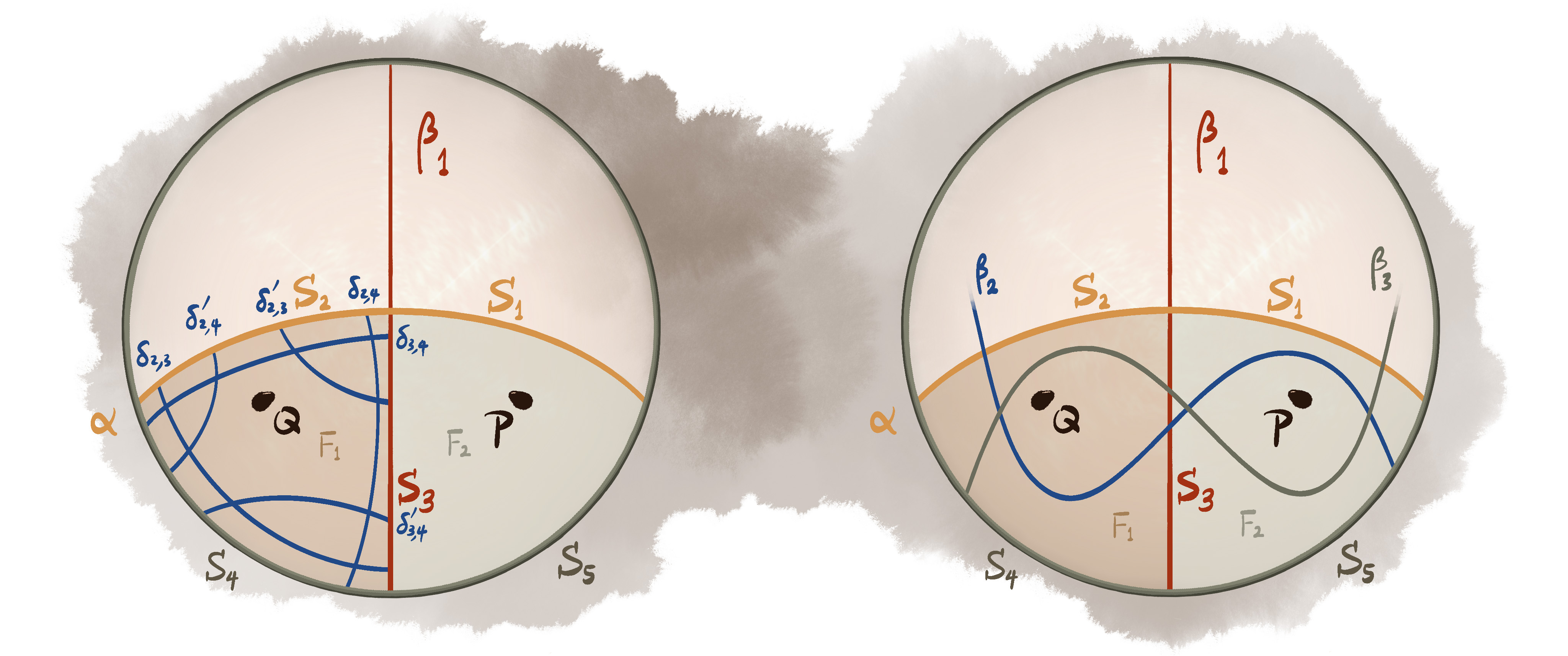}
\caption{The picture for the proof by contradiction. We aim to prove that for any $i, i=1,2,3,4$, there exists an arc in $A$ that is isotopic to $\alpha_i$ on $D_4$, and it is a short arc on $D_5$.}
\label{AtLeastFourShortArcsWhenNIsFive}
\end{figure}

We denote $\Delta_{Q} = \{\delta \mid \delta \text{ is a connected component of } \alpha \cap F_1, \text{ where } \alpha \in A \}$. Clearly, $\Delta_{Q}$ consists of arcs contained in $F_1$, and for any $\delta \in \Delta_{Q}$, both endpoints lie on the boundary $\partial F_1 = S_2 \cup S_3 \cup S_4$. Note that the two endpoints of $\delta$ cannot both lie on $S_2$, since $\delta$ and $\alpha$ cannot intersect more than once. Similarly, the endpoints cannot both lie on $S_3$. Moreover, the endpoints cannot both lie on $S_4$, for otherwise $\delta$ would either be a short arc isolating $Q$ or null-isotopic, both of which lead to a contradiction. In summary, the endpoints of $\delta$ lie on two distinct segments among $S_2$, $S_3$, and $S_4$. Considering isotopies that keep the endpoints of $\delta$ on their respective segments (allowing the endpoints to move freely along these segments), the isotopy type of $\delta$ is limited to exactly six possibilities, as shown in the left picture of \cref{AtLeastFourShortArcsWhenNIsFive}: $\{ \delta_{2,3}, \delta'_{2,3}, \delta_{2,4}, \delta'_{2,4}, \delta_{3,4}, \delta'_{3,4} \}$. This enumeration follows from the fact that $F_1$ is a punctured disk.

We prove that $\Delta_{Q}$ must contain an arc $\delta$ that is isotopic to $\delta_{2,3}$. First, $\delta'_{2,4}$ and $\delta'_{3,4}$ cannot exist because they would each respectively form a half-bigon with $\alpha$ and $\beta_1$. Second, since $\delta_{2,4}$ and $\delta_{3,4}$ intersect at most once, these two arcs cannot fence off $Q$. Therefore, if there were no arcs of type $\delta_{2,3}$ in $\Delta_{Q}$, we could add a short arc isolating $Q$ from $A$, which leads to a contradiction. We denote by $\beta_2$ the arc in $A$ that contains the $\delta$ isotopic to $\delta_{2,3}$.

We analyze the possible configurations of $\beta_2$ within $F_2$. Since $\beta_2$ already intersects $\alpha$ and $\beta_1$ exactly once each, after passing through $S_3$ into $F_2$, $\beta_2$ cannot exit through $S_1$, nor can it intersect $S_3$ again. Therefore, the intersection $\beta_2 \cap F_2$ must be an arc with endpoints lying on $S_3$ and $S_5$, and it has two possible isotopy types: isotopic to either $\delta_{3,5}$ or $\delta'_{3,5}$ (although not explicitly labeled in the figure, these are similar to the cases in $F_1$). The case where $\beta_2$ is isotopic to $\delta'_{3,5}$ is impossible because $\beta_2$ and $\beta_1$ would form a half-bigon. Hence, $\beta_2$ must be as illustrated in the right picture of \cref{AtLeastFourShortArcsWhenNIsFive}.

Performing a similar analysis on $F_2$, in order to fence off $P$, there must be an arc in $A$, denoted by $\beta_3$, as illustrated in the right picture of \cref{AtLeastFourShortArcsWhenNIsFive}. Consequently, $\beta_2$ and $\beta_3$ intersect at least three times, which contradicts the fact that $A$ is a watermelon.
\end{proof}

\begin{lem}
\label{6 arcs P-parallel}
Let $A$ be a maximal watermelon on $D_5$, and let $P$ be a puncture of $D_5$. Suppose that a $P$-reduced watermelon $A'$ of $A$ is the standard maximal watermelon on $D_4$. Then $A$ is watermelon equivalent to the one shown in the right picture of \cref{FourShortArcsWhenNIsFive}.
\end{lem}

\begin{proof}
Since $A'$ is the standard maximal watermelon on $D_4$, we write $A' = \{\alpha'_1, \cdots, \alpha'_6\}$, where $\alpha'_1, \cdots, \alpha'_4$ are the four short arcs on $D_4$. By \cref{lem: $P$-source set has short arc}, $A$ correspondingly contains four short arcs on $D_5$, denoted by $\alpha_1, \cdots, \alpha_4$, each respectively lying in the $P$-source set of $\alpha'_1, \cdots, \alpha'_4$. These arcs respectively isolate the punctures $Q_1, \cdots, Q_4$. We assume $Q_1$ and $Q_3$ are ``opposite'', and $Q_2$ and $Q_4$ are ``opposite'' (see \cref{FourShortArcsWhenNIsFive}).

\begin{figure}[H]
\centering
\includegraphics[width=1.0\textwidth]{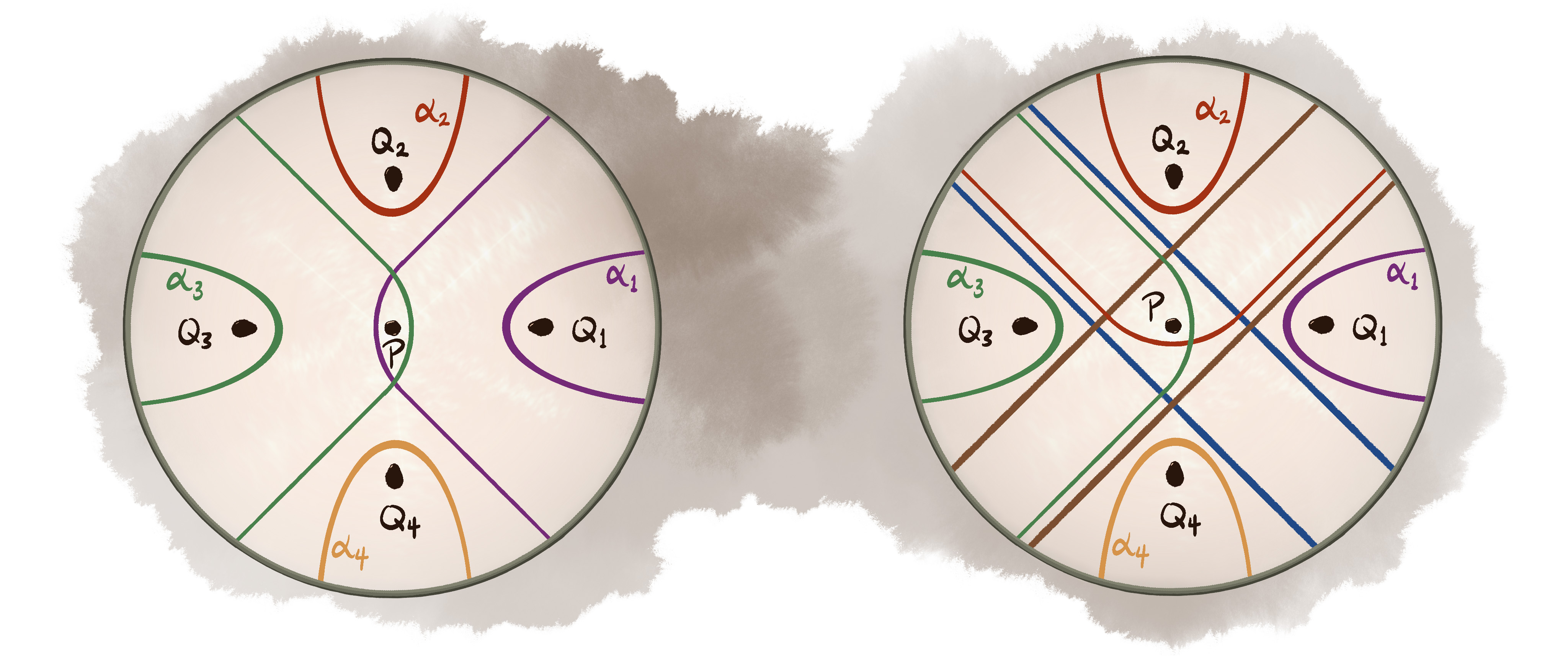}
\caption{This is a maximal watermelon on $D_5$, which has four short arcs.}
\label{FourShortArcsWhenNIsFive}
\end{figure}

Consequently, $D_5 \setminus \cup_{i=1}^{4} F_{\alpha_i,1}$ is topologically a punctured disk, with $P$ as the puncture. Therefore, there is exactly one arc in $A$ that is $P$-parallel to either $\alpha_1$ or $\alpha_3$. Otherwise, suppose there are two such arcs in $A$, one $P$-parallel to $\alpha_1$ and the other $P$-parallel to $\alpha_3$. Then they respectively isolate the pairs $(Q_3, P)$ and $(Q_1, P)$, and must intersect twice (as shown in the left picture of \cref{FourShortArcsWhenNIsFive}), which contradicts the fact that $A$ is a watermelon. On the other hand, suppose that there is no arc in $A$ that is $P$-parallel to $\alpha_1$ or to $\alpha_3$. Then, by \cref{lem: fill point at most two arcs are isotopic}, the set $A$ is partitioned into six $P$-source sets, with the $P$-source sets of $\alpha'_1$ and $\alpha'_3$ each containing only one arc. Since $|A| = 10$, the remaining four $P$-source sets must each contain exactly two arcs. Thus, there must exist two arcs in $A$ such that one is $P$-parallel to $\alpha_2$ and the other is $P$-parallel to $\alpha_4$, and similarly, these two arcs would intersect twice, again yielding a contradiction.

Similarly, there is exactly one arc in $A$ that is $P$-parallel to either $\alpha_2$ or $\alpha_4$. Without loss of generality, we assume that the $P$-source sets of $\alpha'_2$ and $\alpha'_3$ each have cardinality $2$, and that the $P$-source sets of $\alpha'_1$ and $\alpha'_4$ each have cardinality $1$. Consequently, the remaining two $P$-source sets, corresponding to $\alpha'_5$ and $\alpha'_6$, must also each have cardinality $2$, and hence $A$ is equivalent to the watermelon as depicted in the right picture of \cref{FourShortArcsWhenNIsFive}.
\end{proof}

\begin{lem}
\label{n=5 equivalent}
Let $A$ denote the watermelon shown in the right picture of \cref{FourShortArcsWhenNIsFive}. Then the associated complete $1$-system of loops $L_A$ is equivalent to the standard maximal complete $1$-system of loops $L^s$ on $D_5$.
\end{lem}

\begin{proof}
We perform puncture slides of $Q_1$ and $Q_4$ for $L_A$ as in \cref{UniquenessWhenNIsFive}, so that $L_A$ becomes equivalent to $L^s$.

\begin{figure}[H]
\centering
\includegraphics[width=1.0\textwidth]{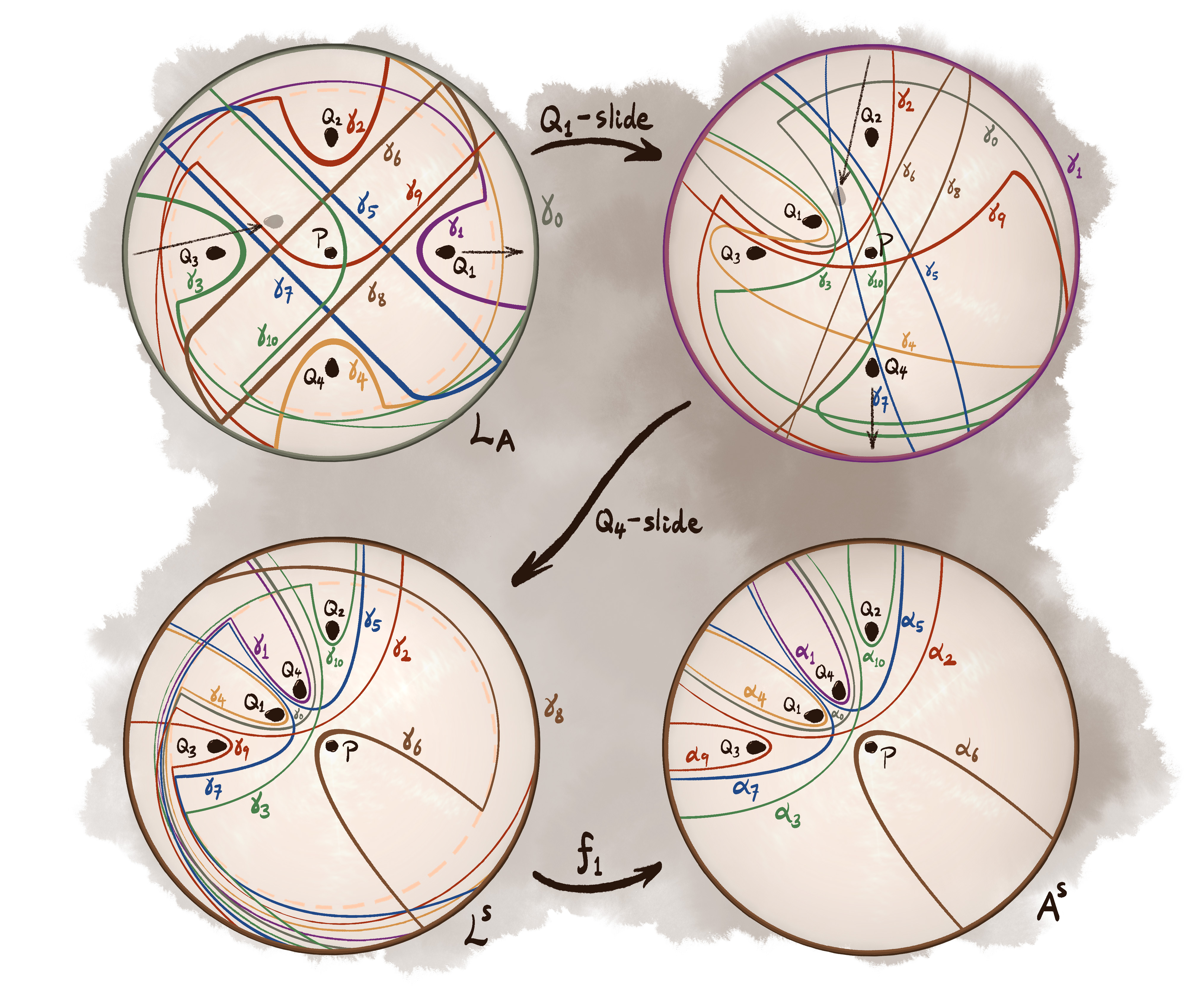}
\caption{The process of performing a puncture slide when $n = 5$. By sliding $Q_2$ and $Q_4$, we verify that $L_A$ and $L'$ are equivalent.}
\label{UniquenessWhenNIsFive}
\end{figure}
\end{proof}

\begin{thm}
\label{Thm: Uniqueness when n = 5}
Maximal complete $1$-systems of loops on $N_{1,5}$ are unique up to equivalence.
\end{thm}

\begin{proof}
Let $A = \{ \alpha_1 , \cdots , \alpha_{10} \}$ be a maximal watermelon on $D_5$. The fact that $|A| = 10$ follows from \cref{cardinality of maximal watermelon}.

If there are $5$ short arcs in $A$, then $A$ is equivalent to the standard maximal watermelon by \cref{Uniqueness of maximal watermelon with n short arcs}.

If there are at most $4$ short arcs in $A$, let $D_4$ be the surface obtained from $D_5$ by filling in a puncture $P$ that is not isolated by any short arc. Then, by \cref{lem: fill point at most two arcs are isotopic}, since at most two distinct arcs in $A$ can be $P$-parallel, any $P$-reduced watermelon $A'$ on $D_4$ must have cardinality at least $5$. If the cardinality of $A'$ is exactly $5$, then it follows from \cref{5 arcs P-parallel} that $A$ is not maximal, yielding a contradiction. Hence the cardinality of $A'$ is exactly $6$, as this is the cardinality of a maximal watermelon on $D_4$, then by the proof of \cref{Thm: Uniqueness when n = 4}, $A'$ must either be the standard maximal watermelon or the maximal watermelon shown in \cref{THreeShortArcsSystemWhenNIsFour}. 

If $A'$ is the standard maximal watermelon on $D_4$, then by \cref{6 arcs P-parallel} and \cref{n=5 equivalent}, we know that $L_A$ is equivalent to the standard maximal complete $1$-system of loops on $D_5$. If $A'$ is the maximal watermelon shown in \cref{THreeShortArcsSystemWhenNIsFour}, then by \cref{lem: $P$-source set has short arc}, each of the three short arcs in $A'$ (denoted by $\alpha'_1, \alpha'_2, \alpha'_3$) has a corresponding short arc on $D_5$ in its $P$-source set (denoted by $\alpha_1, \alpha_2, \alpha_3$ respectively). We now proceed as in the proof of \cref{Thm: Uniqueness when n = 4}, performing a puncture slide $\varphi$ of the puncture isolated by $\alpha_1$ to obtain a complete $1$-system of loops, denoted by $L_\varphi$. The puncture slide $\varphi$ pushes $\alpha_1$ to be isotopic to $\partial D_5$. Let $\varphi(\alpha_1)$ be the mark of $L_\varphi$, and let the associated watermelon be denoted by $A_\varphi$. By \cref{Thm: Uniqueness when n = 4}, $L_{A'}$ (on $D_4$) becomes the standard maximal complete $1$-system of loops on $D_4$ after performing the puncture slide. It follows that the puncture slide transforms the $P$-reduced watermelon of $A$ into the $P$-reduced watermelon of $A_\varphi$, which is the standard maximal watermelon on $D_4$. Applying \cref{6 arcs P-parallel} and \cref{n=5 equivalent} once more, we conclude that $L_\varphi$ is equivalent to the standard maximal complete $1$-system of loops on $D_5$.

\end{proof}

\section{Non-Uniqueness of the Maximal Complete 1-System on Punctured Projective Surfaces with At Least Six Punctures} \label{c4section: Non-Uniqueness of the Maximal Complete One-System on Punctured Projective Surfaces with At Least Six Punctures}

\subsection{Existence of maximal watermelons with n-1 short arcs}
\label{subsection: Existence of maximal watermelons with $n-1$ short arcs}

In \cref{subsection: Existence of maximal watermelons with $n-1$ short arcs}, we will construct a maximal watermelon $A$ with exactly $n-1$ short arcs $(n \geq 4)$.

\begin{thm}
\label{Thm: Existence of max WD with n-1 short arcs}
For every $n \geq 4$, there is a maximal watermelon $A'$ on $D_{n}$ such that $A'$ has exactly $n-1$ short arcs.
\end{thm}

\begin{proof}

Consider the standard maximal watermelon $A^s$ on $D_{n-1}$, where the set of punctures is given by
\begin{equation}
\mathcal{P}=\left\{P_i = \tfrac{1}{2}\exp\left(\frac{2(i-1) \pi \sqrt{-1}}{n-1}\right) \middle| i = 1 , \cdots, n-1 \right\} \notag
\end{equation}
as constructed in \cref{Existence of standard maximal watermelon}.

We denote $u := \exp \left(\frac{2 \pi \sqrt{-1}}{(n+1)(n-1)^2} \right)$. We add one more puncture $P_n$ (chosen so that it avoids all arcs in $A^s$) on $D_{n-1}$, defined by $P_n := \tfrac{1}{2} u^n = \tfrac{1}{2} \exp \left(\frac{2 \pi n \sqrt{-1}}{(n+1)(n-1)^2} \right)$.

Next, we construct $n - 1$ additional arcs $\alpha'_{2,j}$ for $j = 3, \cdots, n$ (with indices understood to be $(n - 1)$-cyclic and symmetric, so that, for instance, $\alpha'_{2,n} = \alpha'_{2,1} = \alpha'_{1,2}$), as well as the arc $\alpha'_{1,3}$ (to be defined below), such that each $\alpha'_{2,j}$ is isotopic on $D_{n-1}$ to the corresponding $\alpha_{2,j}$, and $\alpha'_{1,3}$ is isotopic to $\alpha_{1,3}$ on $D_{n-1}$. For each $j = 3, \cdots, n$, we define
\begin{equation}
z'_{2,j}:= z_{2,j}u = \tfrac{1}{2} \exp{\left( \left( \frac{j-2}{n-1} + \frac{j-2}{(n-1)^2}\right) 2 \pi \sqrt{-1}\right)} u \notag
\end{equation}
and
\begin{equation}
w'_{2,j}:= \tfrac{1}{2} u^{n+2-j}. \notag
\end{equation}
We then define the arcs $\alpha'_{2,j}$ for $j = 3, \cdots, n$ by:
\begin{align}
\label{eq: alpha prime 2 j}
\alpha'_{2,j}(t) :=
\begin{cases}
(2-3t) z'_{2,j} ,& t \in \left[0,\tfrac{1}{3} \right], \\
(2-3t) z'_{2,j} + (3t-1) w'_{2,j} ,& t \in \left[\tfrac{1}{3},\tfrac{2}{3} \right], \\
(3t-1) w'_{2,j} ,& t \in \left[\tfrac{2}{3},1 \right]. \\
\end{cases}
\end{align}
We also define the points
\begin{equation}
z'^+_{1,3} := \tfrac{1}{2} u, \quad z'^-_{1,3} := \tfrac{1}{2} u^{- 1}, \quad w'^+_{1,3} := \tfrac{1}{2} \exp \left(\frac{2\pi \sqrt{-1}}{n-1}\right) u, \quad w'^-_{1,3} := \tfrac{1}{2} \exp \left(\frac{2\pi \sqrt{-1}}{n-1}\right) u^{- 1}. \notag
\end{equation}
Let $v$ denote the intersection point of the two lines $l_{z'^+_{1,3}, z'^-_{1,3}}$ and $l_{w'^+_{1,3}, w'^-_{1,3}}$ (one may verify that $v \in D_n$ for $n \geq 4$). We define the arc $\alpha'_{1,3}$ by
\begin{align}
\label{eq: alpha prime 1 3}
\alpha'_{1,3}(t) :=
\begin{cases}
(2-3t) z'^-_{1,3} ,& t \in \left[0,\tfrac{1}{3} \right], \\
(3-6t) z'^-_{1,3} + (6t-2) v ,& t \in \left[\tfrac{1}{3},\tfrac{1}{2} \right], \\
(4-6t) v + (6t-3) w'^+_{1,3} ,& t \in \left[\tfrac{1}{2},\tfrac{2}{3} \right], \\
(3t-1) w'^+_{1,3} ,& t \in \left[\tfrac{2}{3},1 \right]. \\
\end{cases}
\end{align}
We then define
\begin{equation}
A' := A^s \cup \{\alpha'_{1,j} \mid j = 2, \cdots, n-1\} \cup \{ \alpha'_{1,3} \}, \notag
\end{equation}
Note that $A^s$ is the standard maximal watermelon on $D_{n-1}$, not on $D_n$. We verify that $A'$ is a watermelon on $D_n$ (see \cref{ExistenceOfMaxWatermelonDiagramsWithNMinusOneShortArcs}) as follows. All arcs in $A'$ are straight segments within the disk $\mathbb{B}_{1/2}$, and hence intersect each other at most once inside this disk. Outside $\mathbb{B}_{1/2}$, all arcs except $\alpha'_{1,3}$ are disjoint. The arc $\alpha'_{1,3}$ does not intersect any other arc within $\mathbb{B}_{1/2}$, and intersects each of them at most once outside it. Therefore, the arcs in $A'$ pairwise intersect at most once. Moreover, by inspecting the explicit definitions in \cref{eq: alpha prime 2 j} and \cref{eq: alpha prime 1 3}, together with \cref{ExistenceOfMaxWatermelonDiagramsWithNMinusOneShortArcs}, one can verify that each pair of arcs in $A'$ is in minimal position. Furthermore, $A'$ contains exactly $n - 1$ short arcs, and is maximal, since its cardinality is $\tfrac{1}{2}(n-1)(n-2) + n-1 = \tfrac{1}{2}n(n-1) $.

\begin{figure}[H]
\centering
\includegraphics[width=1.0\textwidth]{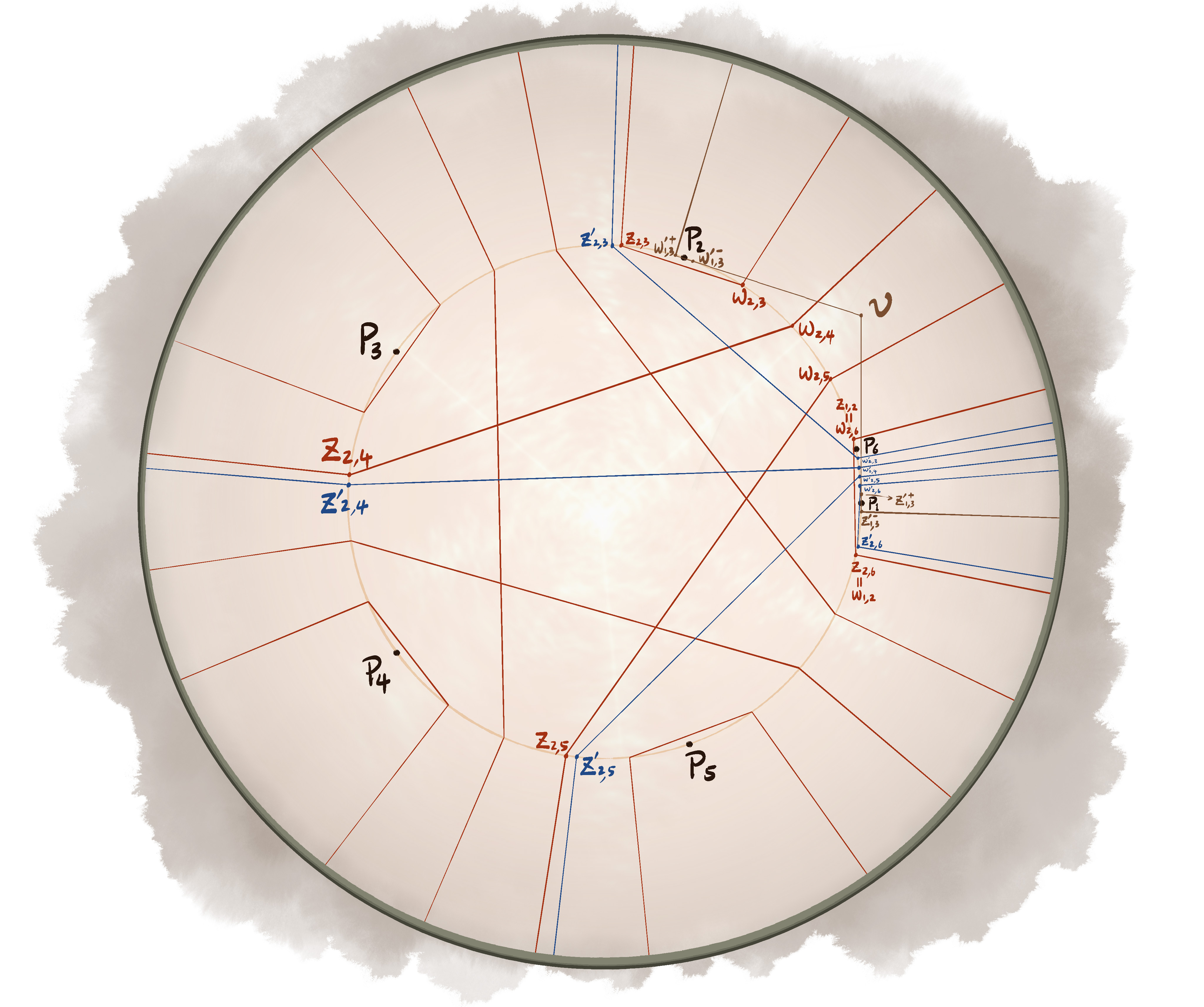}
\caption{A maximal watermelon with $5=6-1$ short arcs.}
\label{ExistenceOfMaxWatermelonDiagramsWithNMinusOneShortArcs}
\end{figure}  

\end{proof}

\subsection{Mod two homology tables}\label{subsection 7.2}

We now prove the non-uniqueness of complete maximal $1$-systems of loops on $D_n / {\sim}$ for $n \geq 6$. Specifically, we show that the $1$-system $L^s$, associated to the standard maximal watermelon $A^s$, is not equivalent to the complete $1$-system $L' := L_{A'}$, where $A'$ is the maximal watermelon constructed in \cref{subsection: Existence of maximal watermelons with $n-1$ short arcs}. In fact, when $n \geq 6$, $L^s$ is not equivalent to any maximal complete $1$-system of loops associated to a maximal watermelon that contains exactly $n - 1$ short arcs.

The main idea is to compare the tables of first $\mathbb{Z}_2$-homology groups associated to the two systems under the action of the mapping class group. We denote the generators of $H_1(N_{1,n}; \mathbb{Z}_2)$ by $g_0, \cdots, g_n$, where $g_0$ corresponds to the loop $\partial D_n / {\sim}$ (that is, the loop $\gamma_0$ in $L$), and $g_1, \cdots, g_n$ correspond to the homology classes of the horocycles $h_i$ encircling the punctures $P_i$, for $i = 1, \cdots, n$ (see \cref{GeneratorsOfZTwoHomologyGroup}). Then $H_1(N_{1,n}; \mathbb{Z}_2) \cong (\mathbb{Z}_2)^n$, with a single relation $g_1 + \cdots + g_n = 0$, so that ${g_0, \cdots, g_{n-1}}$ forms a basis.

Any simple loop $\gamma$ on $D_n / {\sim}$ determines a $\mathbb{Z}_2$-homology class $[\gamma]_2 = \sum_{i=0}^{n} \epsilon_i g_i$ in $H_1(N_{1,n}; \mathbb{Z}_2)$, where each $\epsilon_i = \epsilon_i(\gamma) \in \mathbb{Z}_2$. This representation is not unique, since it is determined only up to the relation $g_1 + \cdots + g_n = 0$. In particular, there are exactly two possibilities for $(\epsilon_0,\ldots,\epsilon_n)$, which differ by flipping each entry $\epsilon_i$ (for $i = 0, \ldots, n$) between $0$ and $1$. We call $(\epsilon_0, \cdots, \epsilon_{n}) \in (\mathbb{Z}_2)^{n+1}$ the \emph{canonical coefficients of $\gamma$} choosen so that $\epsilon_n(\gamma) = 0$, and the \emph{anti-canonical coefficients of $\gamma$} choosen so that $\epsilon_n(\gamma) = 1$.

The entry $\epsilon_0$ is the same in both the canonical and anti-canonical coefficients, and it depends only on the loop $\gamma$: $\epsilon_0(\gamma) = 1$ if and only if $\gamma$ is $1$-sided, and $\epsilon_0(\gamma) = 0$ if and only if $\gamma$ is $2$-sided. According to \cref{lem: max no 2-sided} and \cref{lem: max to complete}, if $L$ is a maximal complete $1$-system of loops on $N_{1,n}$, then every $\gamma \in L$ is $1$-sided, and therefore $\epsilon_0(\gamma) = 1$.

\begin{figure}[H]
\centering
\includegraphics[width=1.0\textwidth]{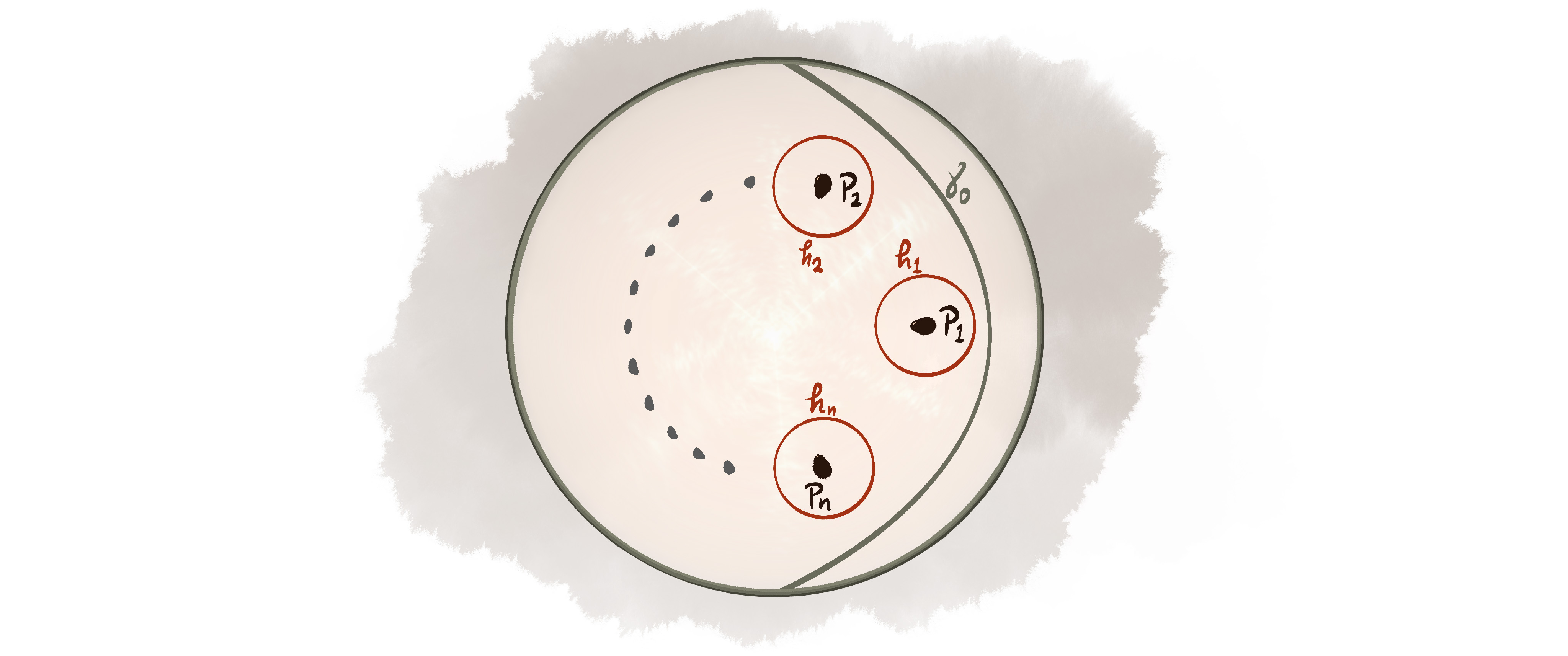}
\caption{The generators of the homology group with $\mathbb{Z}_2$ coefficients of $N_{1,n}$.}
\label{GeneratorsOfZTwoHomologyGroup}
\end{figure}

\begin{defi}[$\mathbb{Z}_2$ homology tables]
For a maximal complete $1$-system $L$ of loops on $N_{1,n}$, every loop $\gamma$ in $L$ determines a $\mathbb{Z}_2$-homology class $[\gamma]_2 = \sum_{i=0}^{n} \epsilon_s g_s$ with canonical coefficients (i.e., $\epsilon_n(\gamma) = 0$). We define the \emph{$\mathbb{Z}_2$ homology table of $L$} as a table with $n+1$ rows labelled by $g_0,\cdots,g_{n}$, and $\tfrac{1}{2}n(n-1)+1$ columns, where each column represents a loop in $L$. The element at the row $g_s$ and column representing $\gamma$ is $\epsilon_s(\gamma) \in \mathbb{Z}_2$. 
\end{defi}

\begin{ex}[The table for standard maximal complete $1$-systems]
\label{standard table example}
Suppose that
\begin{equation}
L^s = \{\gamma_0 = \partial D_n\} \cup \{\gamma_{i,j} := \gamma(\alpha_{i,j}) \mid 1 \leq i < j \leq n\} \notag
\end{equation}
is the standard maximal complete $1$-system of loops on $D_n / {\sim}$ associated to the standard maximal watermelon
\begin{equation}
A^s = \{\alpha_{i,j} \mid 1 \leq i < j \leq n\}, \notag
\end{equation}
where the arcs $\alpha_{i,j}$ are defined as per \cref{Existence of standard maximal watermelon} (see \cref{StandardMaxDiagramFourPunctures} for a depiction).

For the loop $\gamma_0$, we have $\epsilon_0(\gamma_0) = 1$ and $\epsilon_1(\gamma_0) = \cdots = \epsilon_n(\gamma_0) = 0$. For each loop $\gamma_{i,j}$ with $1 \leq i < j \leq n$, we have $[\gamma_{i,j}]_2 = g_0 + \sum_{t=i}^{j-1} g_t$. Therefore, the coefficients of $\gamma_{i,j}$ is given by $\epsilon_0(\gamma_{i,j}) = 1$, $\epsilon_t(\gamma_{i,j}) = 1$ for $i \leq t \leq j - 1$, and $\epsilon_s(\gamma_{i,j}) = 0$ otherwise.

We note the following $n$ specific loops in $L^s$: $\gamma_{i,i+1}$ for $i = 1, \cdots, n - 1$, and $\gamma_{1,n}$. These loops are associated to the short arcs in the standard maximal watermelon. The $\mathbb{Z}_2$-coefficients $\epsilon_i$ of their homology classes follow a distinct pattern. For each $\gamma_{i,i+1}$ with $1 \leq i \leq n - 1$, exactly one of the coefficients $\epsilon_1, \cdots, \epsilon_{n}$ equals $1$ (namely $\epsilon_i = 1$), corresponding to the short arc isolating the puncture $P_i$. For $\gamma_{1,n}$, exactly one of the coefficients $\epsilon_1, \cdots, \epsilon_{n}$ equals $0$ (namely $\epsilon_{n}=0$), since the associated short arc isolates $P_n$. Moreover, these $n$ loops are the only loops in $L$ whose $\mathbb{Z}_2$-homology classes satisfying the condition that either exactly one of the coefficients $\epsilon_1, \cdots, \epsilon_{n}$ equals $1$, or exactly one of them equals $0$. Note that this property is invariant under the ambiguity of the homology class coefficients caused by the relation $g_1 + \cdots + g_n = 0$, that is, under simultaneously flipping all the entries $\epsilon_1, \cdots, \epsilon_n$ between $0$ and $1$.

The $\mathbb{Z}_2$ homology table of $L^s$ is as follows:  

\begin{table}[H]
\centering
\caption{$\mathbb{Z}_2$ homology table of the standard maximal complete $1$-system of loops.}
\label{table of the standard maximal complete $1$-system of loops}
\begin{tabular}{cccccccccccc}
\hline
& $\gamma_0$ & $\gamma_{1,2}$ & $\gamma_{2,3}$ & $\gamma_{3,4}$ & $\cdots$ & $\gamma_{n-2,n-1}$ & $\gamma_{n-1,n}$ & $\gamma_{1,n}$ & $\gamma_{1,3}$ & $\cdots$ & $\gamma_{n-2,n}$ \\
\hline
\mypurple{$g_0$} & \myblue{1} & \myblue{1} & \myblue{1} & \myblue{1} & \myblue{$\cdots$} & \myblue{1} & \myblue{1} & \myblue{1} & \myblue{1} & \myblue{$\cdots$} & \myblue{1}\\
\mypurple{$g_1$} & 0 & \myred{1} & 0 & 0 & $\cdots$ & 0 & 0 & \myred{1} & \mygreen{1} & $\cdots$ & 0\\
\mypurple{$g_2$} & 0 & 0 & \myred{1} & 0 & $\cdots$ & 0 & 0 & \myred{1} & \mygreen{1} & $\cdots$ & 0\\
\mypurple{$g_3$} & 0 & 0 & 0 & \myred{1} & $\cdots$ & 0 & 0 & \myred{1} & 0 & $\cdots$ & 0\\
\mypurple{$\vdots$} & $\vdots$ & $\vdots$ & $\vdots$ & $\vdots$ & \myred{$\ddots$} & $\vdots$ & $\vdots$ & \myred{$\vdots$} & $\vdots$ & $\ddots$ & $\vdots$\\
\mypurple{$g_{n-2}$} & 0 & 0 & 0 & 0 & $\cdots$ & \myred{1} & 0 & \myred{1} & 0 & $\cdots$ & \mygreen{1}\\
\mypurple{$g_{n-1}$} & 0 & 0 & 0 & 0 & $\cdots$ & 0 & \myred{1} & \myred{1} & 0 & $\cdots$ & \mygreen{1}\\
\mypurple{$g_n$} & \myyellow{0} & \myyellow{0} & \myyellow{0} & \myyellow{0} & \myyellow{$\cdots$} & \myyellow{0} & \myyellow{0} & \myyellow{0} & \myyellow{0} & \myyellow{$\cdots$} & \myyellow{0}\\
\hline
\end{tabular}
\end{table}

\end{ex}

\begin{ex}[The table of the maximal $1$-system associated to the watermelon in \cref{Thm: Existence of max WD with n-1 short arcs}]
\label{$n-1$ short arcs table example}
Suppose $L' = L_{A'}$ is the maximal complete $1$-system of loops associated to the maximal watermelon constructed in \cref{Thm: Existence of max WD with n-1 short arcs}. Since $A'$ only has exactly $n-1$ short arcs, there are exactly $n-1$ columns in the table whose coefficients have either exactly one entry equal to $1$ or exactly one entry equal to $0$ in the rows corresponding to the $g_1, \cdots, g_n$.
\end{ex}

\subsection{Proof of non-uniqueness}

Let $L$ be a system of loops on a surface $F$, and $\varphi$ be an element in the mapping class group of $F$, then $\varphi(L) := \{\varphi(\gamma) \mid \gamma \in L\}$ is a system of loops equivalent to $L$. By \cref{MCG generators}, the mapping class group of $N_{1,n}$ is generated by all half-twists and all puncture slides. We fix a collection of generators $g_0 , \cdots, g_n$ of $H_1 (N_{1,n})$ as \cref{subsection 7.2}, and $(\epsilon_0,\cdots,\epsilon_n)(g_0 , \cdots, g_n)^{T}$ is an element in $H_1 (N_{1,n})$.

Let $\tau_{i,j}$ be any half-twist that exchanges the punctures $P_i$ and $P_j$ (with $1 \leq i < j \leq n$). Then $\tau_{i,j}$ induces a transformation $(\tau_{*})_{i,j} : H_1 (N_{1,n} ; \mathbb{Z}_2) \rightarrow H_1 (N_{1,n} ; \mathbb{Z}_2)$, which acts on homology coefficients as follows:
\begin{align}
\label{tau}
(\tau_{*})_{i,j}\left( (\epsilon_0,\cdots,\epsilon_n)
\begin{pmatrix}
g_0 \\
\vdots \\
g_n
\end{pmatrix}
\right) =
(\epsilon_0,\cdots,\epsilon_{i-1},\epsilon_{j}, \epsilon_{i+1},\cdots,\epsilon_{j-1},\epsilon_{i}, \epsilon_{j+1},\cdots,\epsilon_n)\begin{pmatrix}
g_0 \\
\vdots \\
g_n
\end{pmatrix}.
\end{align}

This is because $[\gamma]_2$ can be expressed as a $\mathbb{Z}_2$-linear combination of $[h_1]_2, \dots, [h_n]_2$ and $[\gamma_0]_2$. Although there are infinitely many possible half-twists that exchange the punctures $P_i$ and $P_j$, the specific half-twist $\tau_{i,j}$ only swaps the positions of $h_i$ and $h_j$ within the set $\{h_1, \dots, h_n\}$. Therefore, the induced transformation $(\tau_{*})_{i,j}$ corresponds to exchanging the rows labeled by $g_i$ and $g_j$ in the $\mathbb{Z}_2$-homology table of $L$.

Let $\sigma_i$ be any puncture slide that slides the puncture $P_i$. Then $\sigma_i$ induces a transformation $(\sigma_*)_i : H_1(N_{1,n}; \mathbb{Z}_2) \to H_1(N_{1,n}; \mathbb{Z}_2)$, which acts on homology coefficients as follows:
\begin{align}
\label{sigma}
(\sigma_*)_{i}\left( (\epsilon_0,\cdots,\epsilon_n)\begin{pmatrix}
g_0 \\
\vdots \\
g_n
\end{pmatrix} \right) =
(\epsilon_0,\cdots,\epsilon_{i-1},\epsilon_{i}+\epsilon_0, \epsilon_{i+1},\cdots, \epsilon_n)\begin{pmatrix}
g_0 \\
\vdots \\
g_n
\end{pmatrix}.
\end{align}

Similarly, although there are infinitely many puncture slides that slide the puncture $P_i$, the induced transformation $(\sigma_{*})_i$ always corresponds to adding the $g_0$-row to the $g_i$-row in the $\mathbb{Z}_2$-homology table.

By \cref{MCG generators}, the mapping class group of $N_{1,n}$ is generated by all half-twists and puncture slides. Hence, any element in the mapping class group induces an invertible linear transformation on the $\mathbb{Z}_2$-homology table of $L$: this transformation consists of exchanging some of the rows indexed by $g_1, \cdots, g_n$ (i.e., excluding the first row corresponding to $g_0$), and adding the first row to other rows.

\begin{defi}[$\mathbb{Z}_2$-homological differences]
Let $\gamma_1$ and $\gamma_2$ be two simple loops on $N_{1,n}$. Define
\begin{equation}
\delta'_2(\gamma_1, \gamma_2) := \# \left\{ i \in \{1,\cdots,n\} \mid \epsilon_i(\gamma_1) \neq \epsilon_i(\gamma_2) \right\}, \notag
\end{equation}
and define
\begin{equation}
\delta_2(\gamma_1, \gamma_2) := \min\{\delta'_2(\gamma_1, \gamma_2) , n-\delta'_2(\gamma_1, \gamma_2)  \}.\notag
\end{equation}

Note that although the value of $\delta'_2(\gamma_1, \gamma_2)$ depends on the choice of coefficients for $\gamma_1$ and $\gamma_2$ (canonical or anti-canonical), the quantity $\delta_2(\gamma_1, \gamma_2)$ is independent of this choice. This is because flipping the coefficients of either $\gamma_1$ or $\gamma_2$ replaces $\delta'_2$ with $n - \delta'_2$. We call $\delta_2(\gamma_1, \gamma_2)$ the \emph{$\mathbb{Z}_2$-homological difference} between $\gamma_1$ and $\gamma_2$; this is an invariant independent of the choice of coefficients.
\end{defi}

\begin{lem}
\label{lem: delta_2 never changed}
Let $\gamma_1$ and $\gamma_2$ be two simple loops on $N_{1,n}$. For any element $\varphi$ of the mapping class group, we have
\begin{equation}
\delta_2(\gamma_1, \gamma_2) = \delta_2(\varphi(\gamma_1), \varphi(\gamma_2)). \notag
\end{equation}
\end{lem}

\begin{proof}
Since $\varphi$ is generated by half-twists and puncture slides, it induces transformation as described in \cref{tau} and \cref{sigma}. Both types of transformation preserve $\delta_2$. Moreover, simultaneously flipping all entries $\epsilon_i$ for $i = 1, \cdots, n$ of both loops does not change $\delta_2$.
\end{proof}

\begin{defi}[Sets of relative short loops]
\label{defi: Sets of relative short loops}
Let $L$ be a complete $1$-system of loops on $N_{1,n}$, and let $\gamma$ be a loop in $L$. We define the \emph{Set of relative short loops of $\gamma$} to be
\begin{equation}
S(\gamma) := \{ \beta \in L \mid \delta_2(\gamma, \beta) = 1 \}.
\end{equation}
We call any $\beta \in S(\gamma)$ a \emph{relative short loop of $\gamma$}.

The geometric meaning of a relative short loop is as follows: Fix $\gamma$ as a mark of $L$ (with an arbitrary orientation). Let $A_L$ denote the watermelon associated with the marked complete $1$-system of loops $(L, \overrightarrow{\gamma})$. Then the loops in $S(\gamma)$ are associated with the short arcs in $A_L$.
\end{defi}

\begin{cor}
\label{cor: S varphi gamma is varphi S gamma}
Combining \cref{lem: delta_2 never changed} with \cref{defi: Sets of relative short loops}, we obtain the following:  
If $L_1$ and $L_2$ are maximal complete $1$-systems of loops on $N_{1,n}$ with $n \geq 2$, and if $\varphi$ is a representative of a mapping class such that $\varphi(L_1)$ is isotopic to $L_2$ (up to isotopy, we may assume $L_2 = \{\varphi(\gamma) \mid \gamma \in L_1\}$), then for any $\gamma \in L_1$, we have

\begin{equation}
S(\varphi(\gamma)) = \{\varphi(\beta) \mid \beta \in S(\gamma)\} =: \varphi(S(\gamma)). \notag
\end{equation}
\end{cor}    

\begin{thm}[Non-uniqueness for $n \geq 6$]
\label{thm: NonUniquenessNGeqSix}
For $n \geq 6$, the maximal complete $1$-systems of loops on $N_{1,n}$ are not unique up to the action of the mapping class group.
\end{thm}

\begin{proof}
The idea is to show that the $\mathbb{Z}_2$-homology table of the standard maximal complete $1$-system of loops $L^s$ in \cref{standard table example} cannot be transformed into the table of the maximal complete $1$-system of loops $L'$ in \cref{$n-1$ short arcs table example} by any mapping class group action.

Assume that $\varphi$ is a representative of a mapping class such that $\varphi(L^s)$ is isotopic to $L'$ (see \cref{defi: Equivalent systems}). We will show that this is impossible.

Since $L'$ contains the loop $\gamma'_0$ satisfying $[\gamma'_0]_2 = g_0$ (note that $\gamma'_0$ is actually the same as $\gamma_0$, but we use a prime to distinguish it as an element of $L'$), there exists a loop $\gamma \in L^s$ such that $\varphi_{*}([\gamma]_2) = [\gamma'_0]_2 = g_0$.

Consider the loops in $L'$ corresponding to the $n-1$ short arcs, denoted by $\gamma'_1, \cdots, \gamma'_{n-1}$; that is, the loops whose coefficients have exactly one $1$ or exactly one $0$ among the entries in the rows labeled $g_1, \cdots, g_n$. In other words, we have $S(\gamma'_0) = \{\gamma'_1, \cdots, \gamma'_{n-1}\}$. Then there exist loops $\beta_1, \cdots, \beta_{n-1}$ in $L^s$ such that $\varphi(\beta_i) = \gamma'_i$ for each $i = 1, \cdots, n-1$. By \cref{cor: S varphi gamma is varphi S gamma}, it follows that $S(\gamma) = \varphi^{-1}(S(\gamma'_0)) = \{\beta_1, \cdots, \beta_{n-1} \}$, so in particular, we have $\# S(\gamma) = n-1$.

Now observe the structure of the $\mathbb{Z}_2$-homology table of $L^s$:

\begin{itemize}
\item The first entry (corresponding to $g_0$) of every column is $1$;
\item The last entry (corresponding to $g_n$) is $0$;
\item The middle entries (rows $g_1$ to $g_{n-1}$) form a block pattern of consecutive $0$s, followed by consecutive $1$s, followed again by consecutive $0$s.
\end{itemize}

Given this, for the columns corresponding to $\gamma$ and $\beta_i$ to differ in exactly $1$ or $n-1$ positions, we proceed by case analysis, depending on whether $\gamma = \gamma_0$ or $\gamma = \gamma_{i,j}$ for some $1 \leq i < j \leq n$. In all of the following cases, we will show that $\# S(\gamma) = n$ or $\leq 4$. This contradicts the previously established fact that any loop $\gamma$ in $L$ must satisfy $\# S(\gamma) = n - 1 \geq 5$.

If $\gamma$ is isotopic to $\gamma_0$, then in $L$, there are exactly $n$ loops, namely $\{\gamma_{i,i+1} \mid i = 1, \cdots, n-1\} \cup \{\gamma_{1,n}\}$, that satisfy $\delta_2(\gamma_0, \beta) = 1$ or $n-1$ for each $\beta$ in the set. That is, $S(\gamma) = S(\gamma_0) = \{\gamma_{i,i+1} \mid i = 1, \cdots, n-1\} \cup \{\gamma_{1,n}\}$. This implies that $\# S(\gamma) = n > n-1$.

If $\gamma$ is isotopic to $\gamma_{i,j}$ with $1 < i < j < n$, then:
\begin{itemize}
\item If $j > i + 1$, we have $S(\gamma) = \{\gamma_{i-1,j},\gamma_{i+1,j},\gamma_{i,j-1}, \gamma_{i,j+1}\}$.
\item If $j = i + 1$, we have $S(\gamma) = \{\gamma_0, \gamma_{i-1,j}, \gamma_{i,j+1}\}$.
\end{itemize}
In both cases, $\# S(\gamma) \leq 4 < n - 1$.

If $\gamma$ is isotopic to $\gamma_{i,j}$ with $1=i<j<n$, then:
\begin{itemize}
\item If $j > i + 1$, we have $S(\gamma) = \{\gamma_{2,j}, \gamma_{1,j-1}, \gamma_{1,j+1}, \gamma_{j,n}\}$.
\item If $j = i + 1$, we have $S(\gamma) = \{\gamma_0, \gamma_{1,3}, \gamma_{2,n}\}$.
\end{itemize}
Again, $\# S(\gamma) \leq 4$.

If $\gamma$ is isotopic to $\gamma_{i,j}$ with $1<i<j=n$, then:
\begin{itemize}
\item If $j > i + 1$, we have $S(\gamma) = \{\gamma_{i,n-1}, \gamma_{i-1,n}, \gamma_{i+1,n}, \gamma_{1,i} \}$.
\item If $j = i + 1$, we have $S(\gamma) = \{ \gamma_0, \gamma_{n-2,n}, \gamma_{1,n-1}\}$.
\end{itemize}
So $\# S(\gamma) \leq 4$.

If $\gamma$ is isotopic to $\gamma_{1,n}$, then $S(\gamma) = \{\gamma_0, \gamma_{1,n-1}, \gamma_{2,n}\}$, hence $\# S(\gamma) = 3 \leq 4$.

Therefore, no mapping class representative $\varphi$ satisfies that $\varphi(L^s)$ is isotopic to $L'$. It follows that $L'$ and $L^s$ are not equivalent maximal complete $1$-systems of loops. Hence, for $n \geq 6$, the maximal complete $1$-system of loops on $N_{1,n}$ is not unique up to the action of the mapping class group.

\end{proof}

\section*{Acknowledgement}
    At the end of this paper, We would like to express our gratitude to Yi Huang for his invaluable assistance and guidance. We also thank Weiyan Chen for supporting the first listed author's trip to the US, where these results were presented, and Yitwah Cheung for providing valuable feedback on the first listed author's thesis proposal, which included some results from this paper.

    % I also extend my thanks to Makoto Sakuma, as well as all the friends and colleagues in Japan who listened to my report, for their valuable suggestions and questions. Finally, I want to express my appreciation to Zhulei Pan, Wujie Shen, Ivan Telpukhovskiy and Sangsan (Tee) Warakkagun as some of the inspiration for this paper stemmed from discussions with them.

%==============================%
\newpage

% \bibliographystyle{acm}
% \bibliography{reference.bib}
\printbibliography[title={References}] %打印参考文献

%\end{multicols}

\end{document}

%% file: preamble.tex
\usepackage[a4paper,margin=1.3in]{geometry}
\usepackage{amsmath}           % AMS 的主包
\usepackage{amssymb}           % AMS 符号包，会自动加载 amsfonts
\usepackage{mathrsfs}          % \mathscr 字体样式
\usepackage{eucal}             % 更改 \mathcal 的字体样式
\usepackage{amsthm}            % AMS 的定理环境
\usepackage{mathtools}         % 一些额外的数学环境功能，包括 \dfrac
\usepackage{stmaryrd}          % 更更多的特殊符号
\usepackage{arcs}              % 弧線hat
\usepackage{esint}             % 更多积分符号
\usepackage{extarrows}         % 提供在箭头上下写字的命令
\usepackage{enumerate}         % 带序号列表环境
\usepackage{xcolor}            % 让 latex 支持花里胡哨的颜色的包
\usepackage{graphicx}          % 插入各种类型图片支持
\usepackage{tikz}              % tikz 绘图环境
\usepackage{tikz-3dplot}       % 一个简单的 tikz 3d 绘图宏包
\usepackage{tikz-cd}           % 基于 tikz 的交换图表绘制工具
\usepackage[all,cmtip]{xy}     % 交换图表绘制工具，命令为 \xymatrix
\usepackage{pgfplots}          % 高级绘图包，基于 tikz，包括 2d 和 3d 绘图
\pgfplotsset{compat=newest}    % 设置 pgfplots 兼容性选项，以使用新功能
\usepackage{subcaption}        % 可以交叉引用的图表并排环境
\usepackage[ocgcolorlinks,linkcolor=blue]{hyperref}
\usepackage[capitalise]{cleveref}          
\usepackage[hyperref=true,backend=biber,style=alphabetic,backref=true,url=false]{biblatex}
\usepackage{tcolorbox}         % 绘制彩色文本框的宏包
\tcbuselibrary{most}           % 加载 tcolorbox 的库
\usepackage{bm}                % 为所有数学字体添加粗体，命令 \bm{abcd}
\usepackage{slashed}           % 在符号上添加反划线，命令 \slashed

\usepackage[utf8]{inputenc}
\usepackage[T1]{fontenc}

\usepackage{svg}               % 插入svg圖

\usepackage{multicol}          %用于实现在同一页中实现不同的分栏

\usepackage{amsmath,amsthm,amsfonts,amssymb,bm}

\usepackage{fourier}

\usepackage{bbold}

\renewcommand{\emptyset}{\varnothing}

\usepackage{graphicx} %插入图片的宏包
\usepackage{float} %设置图片浮动位置的宏包

\theoremstyle{plain}\newtheorem{thm}{Theorem}[section]
\theoremstyle{definition}\newtheorem{defi}[thm]{Definition}
\theoremstyle{plain}
\theoremstyle{plain}\newtheorem{cor}[thm]{Corollary}
\theoremstyle{plain}\newtheorem{lem}[thm]{Lemma}
\theoremstyle{plain}
\theoremstyle{plain}\newtheorem{pro}[thm]{Proposition}
\theoremstyle{plain}
\theoremstyle{plain}
\theoremstyle{plain}
\theoremstyle{remark}
\theoremstyle{remark}
\theoremstyle{remark}\newtheorem{nota}[thm]{Notation}
\theoremstyle{remark}
\theoremstyle{plain}
\theoremstyle{plain}\newtheorem{ex}[thm]{Example}
\theoremstyle{plain}

\theoremstyle{remark}\newtheorem{rmk}[thm]{Remark}
\theoremstyle{remark}
\theoremstyle{remark}
\numberwithin{equation}{section}
\numberwithin{thm}{section}

\usepackage[backend=biber,style = alphabetic]{biblatex}
\usepackage{xcolor}

\definecolor{c1}{HTML}{88d5d2} %淡蓝
\definecolor{c2}{HTML}{9c9d47} %橄榄绿
\definecolor{c3}{HTML}{fec842} %鹅蛋黄
\definecolor{c4}{HTML}{e97a2e} %霞光黄
\definecolor{c5}{HTML}{834e71} %月光紫

\usepackage{xcolor}
\definecolor{myred}{HTML}{a43113}
\definecolor{myblue}{HTML}{1f4887}
\definecolor{mygreen}{HTML}{49804c}
\definecolor{mygraygreen}{HTML}{6b6c5c}
\definecolor{myyellow}{HTML}{d5934a}
\definecolor{mypurple}{HTML}{492f7a}

\newcommand{\myred}[1]{\textcolor{myred}{#1}}
\newcommand{\myblue}[1]{\textcolor{myblue}{#1}}
\newcommand{\mygreen}[1]{\textcolor{mygreen}{#1}}

\newcommand{\myyellow}[1]{\textcolor{myyellow}{#1}}
\newcommand{\mypurple}[1]{\textcolor{mypurple}{#1}}